\documentclass[reqno]{amsart}

\usepackage{amsthm}
\usepackage{amsmath}
\usepackage{amssymb}
\usepackage{mathrsfs}
\usepackage{bbm}
\usepackage{amscd}
\usepackage{extarrows}
\usepackage[colorlinks]{hyperref}

\theoremstyle{plain}
\newtheorem{thm}{Theorem}[section]
\newtheorem{thml}{Theorem}

\newtheorem{prop}[thm]{Proposition}
\newtheorem{lem}[thm]{Lemma}

\newtheorem{cor}[thm]{Corollary}

\newtheorem*{mainthm'}{Main Theorem'}
\newtheorem*{clm}{Claim}

\theoremstyle{definition}
\newtheorem{defn}[thm]{Definition}
\newtheorem{exa}[thm]{Example}

\theoremstyle{remark}
\newtheorem*{rmk}{Remark}
\newcommand{\eps}{\varepsilon}

\newcommand{\whT}{\widehat{T}}

\newcommand{\dif}{\,\mathrm{d}}
\newcommand{\Dif}{\mathrm{D}}

\newcommand{\A}{\mathbb{A}}
\newcommand{\R}{\mathbb{R}}
\newcommand{\Z}{\mathbb{Z}}

\newcommand{\Crit}{\mathrm{Crit}}

\newcommand{\ind}{1}

\DeclareMathOperator{\diam}{diam} 
\DeclareMathOperator*{\Ls}{Ls}

\title[Periodic optimization over an expanding circle map]{Typicality of periodic optimization over an expanding circle map}

\author[R. Gao]{Rui Gao}

\address{Rui Gao: School of Mathematical Sciences, Qufu Normal University, Jining 273165, China}
\email{gaoruimath@qfnu.edu.cn}

\author[W. Shen]{Weixiao Shen}
\address{Weixiao Shen: Shanghai Center for Mathematical Sciences, New Cornerstone Science Laboratory, Jiangwan Campus, Fudan University, No
	2005 Songhu Road, Shanghai 200438, China}
\email{wxshen@fudan.edu.cn}

\author[R. Zhang]{Ruiqin Zhang}
\address{Ruiqin Zhang: School of Mathematical Sciences, Fudan University, No 220 Handan Road, Shanghai 200433, China}
\email{21110180029@m.fudan.edu.cn}
\begin{document}

	\begin{abstract}
		We study the ergodic optimization problem over a real analytic expanding circle map. 
		We show that in both the topological and the measure-theoretical senses, a typical $C^r$ performance function has a unique maximizing measure 
		and the unique maximizing measure is supported on a periodic orbit, for $r=1,2,\dots,\infty,\omega$.
	\end{abstract}
	
	\maketitle
	
\section{Introduction}\label{sec:intro} 
	Given a continuous map $T: X\to X$ from a non-empty compact metric space into itself, let $\mathcal{M}(T)$ denote the collection of  all $T$-invariant Borel probability measures. For a continuous function $f:X\to \mathbb{R}$,  which is called {\em a performance function} in the following, let
	\begin{equation}\label{eqn:betaalpha}
		\beta(f,T)=\sup_{\mu\in\mathcal{M}(T)} \int f \dif\mu.
	\end{equation}
	A measure $\mu\in\mathcal{M}(T)$ is called a {\em maximizing measure} if $\beta(f,T)=\int f \dif\mu$.   
	An outstanding problem in {\em ergodic optimization} is the so-called {\em typically periodic optimization} ({\em TPO} for short) problem: whether typically a maximizing measure is a {\em periodic measure} of $T$, i.e., supported on a periodic orbit, for chaotic dynamical systems $T$ (such as uniformly expanding maps, or Axiom A diffeomorphisms). TPO was first conjectured by Hunt-Ott \cite{HO96L, HO96E} to hold in a measure-theoretical sense for typical chaotic systems, and then by Yuan-Hunt \cite{YH96} in a topological sense for Axiom A or uniformly expanding systems. See also \cite{Boc18,Jen19}. 

As evidence to support these conjectures, in the case that $T(x)=2x \mod 1$, Bousch~\cite{Bou00} studied the specific one-parameter family $f_t(x)=\cos (2\pi (x-t))$ and proved that for almost every $t$, $f_t$ is uniquely maximized by a periodic measure. Indeed, he showed that for every $t$, $f_t$ has a unique maximizing measure and the measure is {\em Sturmian}, i.e., supported in a semicircle.  See~\cite{Jen07, Jen08, FSS21, Gao21} for similar results in this direction.

The topological TPO problem has a negative answer in the $C^0$ category: for a general topological dynamical system $T:X\to X$, a generic continuous $f:X\to\R$ has a unique maximizing measure and the measure is of zero entropy but fully supported; see \cite{BJ02, JM08, Bre08, Mor10}.
	It has a positive answer for the space of H\"older functions, Lipschitz functions and $C^1$ functions,  in the case that $T$ is uniformly expanding or Axiom A; see~\cite{CLT01, Con16,HLMXZ19}. However, the perturbation techniques used in these works (originated from ~\cite{YH96}) cannot be applied to spaces of performance functions with higher regularity. 

	There have been very few results on the measure-theoretical typicality in ergodic optimization. In \cite{Mor21}, it is proved that in a very general setting, having a unique maximizing measure is a prevalent property; see also \cite{JLLZ25}. In \cite{BZ16, DLZ24}, in the case that $T$ is the full one-sided shift on two symbols, it is proved that for spaces of functions with a very strong modulus of continuity, periodic maximizing is a prevalent property in these spaces (similar spaces were studied earlier in \cite{QS12} for the topological TPO problem). However, the required modulus of continuity is too strong and there is no natural analogy for functions defined on a manifold.

	This paper is devoted to studying the TPO problem for expanding circle maps. A $C^1$ map $T:\R/\Z\to\R/\Z$ is {\em expanding} if there exist $C_*>0$ and $\lambda_*>1$ such that 
	\begin{equation}\label{eqn:C*lambda*}
		|(T^n)'(x)|\ge C_*\lambda_*^n
	\end{equation}
	holds for all $x\in \R/\Z$ and all positive integers $n$. For $r=1,2,\dots, \infty,\omega$, we use $\mathcal{E}^r$ to denote the space of all $C^r$ expanding circle maps and use $\mathcal{C}^r$ to denote the space of all $C^r$ functions from $\R/\Z$ to $\R$. Here $C^\omega$ means real analytic as usual. When $r=1,2,\dots, \infty$,  $\mathcal{E}^r$ and $\mathcal{C}^r$ are endowed with the $C^r$ topology. Moreover, for each $q=1,2,\dots$, we say that $f_{\mathbf{t}}:\R/\Z\to \R$, $\mathbf{t}\in [0,1]^q$, is {\em a $C^r$ family with $q$ parameters} if $(\mathbf{t}, x)\mapsto f_{\mathbf{t}}(x)$ is $C^r$; the space of all $C^r$ families with $q$ parameters will also be endowed with the $C^r$ topology. 

Let $d(\cdot,\cdot)$ denote the natural metric on $\R/\Z$ throughout the paper. For $T\in \mathcal{E}^1$ and $f\in \mathcal{C}^1$, the {\em Aubry set} of  $(f,T)$ (this notion will be reviewed in \S~\ref{sse:aubry}), denoted by $\Omega(f,T)$, is defined as the collection of $x\in\R/\Z$ satisfying the following property: for any $\eps>0$, there exists a positive integer $n$ and $z\in T^{-n}(x)$ such that $d(z,x)<\eps$ and
	$$\Big|\sum_{i=0}^{n-1}\left(f(T^i(z))-\beta(f,T)\right)\Big|<\eps.$$
It is not difficult to deduce from~\cite{CLT01} that (see Corollary~\ref{cor:lock}) if $\Omega(f,T)$ is a periodic orbit of $T$, then $(f,T)$ satisfies the following $C^1$ {\em locking property}: each function in some $C^1$ open neighborhood of $f$ has a unique maximizing measure and it is supported on the periodic orbit $\Omega(f,T)$.

	\begin{thml}\label{thm:A} Let $r=1,2,\dots, \infty$ and let $q=1,2,\dots$. Then there is a residual subset $\mathcal{E}_*^r$ of $\mathcal{E}^r$ which contains $\mathcal{E}^\omega$ such that for any $T\in \mathcal{E}_*^r$ the following holds for a generic $C^r$ family of functions $f_{\mathbf{t}}: \R/\Z\to \R$, $\mathbf{t}\in [0,1]^q$: For Lebesgue a.e. $\mathbf{t}\in [0,1]^q$, $\Omega(f_{\mathbf{t}},T)$ is a periodic orbit.
	\end{thml}

	This provides a confirmation of Hunt-Ott's conjecture \cite{HO96L,HO96E} for expanding circle maps with $C^r$ performance functions, $r=1,2,\dots,\infty$.

\begin{thml}\label{thm:B} Let $r=1,2,\dots,\infty$. Then there is a residual subset $\mathcal{E}_*^r$ of $\mathcal{E}^r$ which contains $\mathcal{E}^\omega$ such that the following holds for each $T\in \mathcal{E}_*^r$: There is an open and dense subset $\mathcal{D}^r(T)$ of $\mathcal{C}^r$ such that for each $f\in \mathcal{D}^r(T)$, $\Omega(f,T)$ is a periodic orbit.
	\end{thml}

Yuan-Hunt's conjecture asserts that we can take $\mathcal{E}_*^r=\mathcal{E}^r$ in Theorem~\ref{thm:B}. For $r=1$, this was proved in~\cite{HLMXZ19}. 
For $r=2,3,\dots,\infty$, Theorem~\ref{thm:B} provides strong support for Yuan-Hunt's conjecture in the $C^r$ topology, but the conjecture itself remains open.

Both theorems are deduced from Theorem~\ref{thm:main} below, using the density of $\mathcal{E}^\omega$ in $\mathcal{E}^r$ for each $r=1,2,\dots,\infty$. In Theorem~\ref{thm:main}, we prove that for $T\in \mathcal{E}^\omega$, in an appropriate space of $C^\omega$ performance functions, being uniquely maximized by a periodic measure is a prevalent property.  Working in the real analytic setting brings a key advantage: every real analytic performance function has a maximizing measure of entropy zero, as proved in~\cite{GS24}.

	Prevalence is a concept introduced in~\cite{HSY92} to address measure-theoretical typicality for infinite-dimensional spaces (it was recognized later by the authors of ~\cite{HSY92} that this notion was already introduced before in \cite{Chr73}). Let $V$ be a topological vector space which is completely metrizable. A subset $S$ of $V$ is called {\em shy} (or a {\em Haar zero set} following \cite{Chr73}) if there is a Borel probability measure on $V$ with compact support and a Borel set
	$\widetilde{S}\supset S$ such that 
	$$\mu(\widetilde{S}-v)=0, \text{ for all } v\in V.$$
	In this case, we say that $\mu$ is a {\em transverse} measure for $S$.
	A subset of $V$ is called {\em prevalent} if its complement is shy. We recommend~\cite{HK10} for many applications of this notion in dynamical systems.

Given a linear subspace $\mathcal{X}$ of $\mathcal{C}^1$ and $T\in\mathcal{E}^1$, denote
\begin{equation}\label{eqn:sharp}
  \mathcal{X}_{\#}(T)=\{f\in \mathcal{X}: \text{$\Omega(f,T)$ is a periodic orbit}\}.
\end{equation}
Denote $\mathbb{N}=\{0,1,\dots\}$ and $\mathbb{N}_1=\{1,2,3,\dots\}$. The main result of this paper is as follows. 

	\begin{thml}\label{thm:main} 
Let $T\in\mathcal{E}^\omega$ and let $\mathcal{H}\subset \mathcal{C}^\omega$ be a Banach space containing all trigonometric polynomials. Then $\mathcal{H}_{\#}(T)$ is a prevalent subset of $\mathcal{H}$. More precisely, there exist $\varphi_n\in \mathcal{H}$, $n=1,2,\dots$ such that $\sum_{n=1}^\infty \|\varphi_n\|_{\mathcal{H}}<\infty$ and such that for each $f\in \mathcal{H}$, 
$$m_\infty\left(\left\{(t_n)_{n=1}^\infty\in [0,1]^{\mathbb{N}_1}: f+ \sum_{n=1}^\infty t_n \varphi_n \not\in \mathcal{H}_{\#}(T)\right\}\right)=0,$$
where $m_\infty$ denotes the infinite product of the Lebesgue measure on $[0,1]$. 
	\end{thml}

	The following Banach spaces $\mathcal{C}^\omega_{p,\rho}$, $1\le p\le \infty$ and $\rho>0$, serve as concrete examples which satisfy the assumptions of Theorem~\ref{thm:main}.
	Let $$U_\rho=\{x+\mathrm{i} y: x\in \R, |y|<\rho\}/\Z\subset \mathbb{C}/\mathbb{Z},$$
	and let $\mathcal{A}^p(U_\rho)$ (known as the {\em Bergman space}) denote the collection of all $L^p$ 
	holomorphic functions $F:U_\rho\to \mathbb{C}$, endowed with the $L^p$ norm. Let 
\begin{equation}\label{eqn:analytic space}
  \mathcal{C}^\omega_{p,\rho}=\{f\in \mathcal{C}^\omega: f \text{ extends to a function in } \mathcal{A}^p(U_\rho)\},
\end{equation}
endowed with the norm induced from $\mathcal{A}^p(U_\rho)$. Each of the spaces $\mathcal{C}^\omega_{p,\rho}$ obviously contains all trigonometric polynomials. We remark that the inclusion from $\mathcal{C}^\omega_{p,\rho}$ to $\mathcal{C}^r$ is continuous, $r=0,1,\dots,\infty$, where $\mathcal{C}^0$ denotes the space of continuous functions from $\R/\Z$ to $\R$ endowed with the $C^0$ topology. The case $r=0$ follows from for example \cite[Theorem 1, p.7]{DS04}. Using the Cauchy formula, the rest cases follows.

	{\bf Outline of the proof.} We shall now explain our strategy for the proof of Theorem~\ref{thm:main}. Fix $T\in\mathcal{E}^\omega$ and fix a Banach space $\mathcal{H}$ as in Theorem~\ref{thm:main}. Let $\mathcal{H}_{\mathrm{per}}$ denote the collection of $f\in \mathcal{H}$ admitting a periodic maximizing measure.

For each $f\in \mathcal{C}^1$, one can associate to $f$ a $T$-invariant compact subset $K(f)=K(f,T)$ such that $K(f)$ depends on $f\in \mathcal{C}^1$ in an upper semi-continuous manner and such that the $(f,T)$-maximizing measures are precisely the measures in $\mathcal{M}(T)$ supported on $K(f)$  
(see Proposition~\ref{prop:setK}). A sufficient condition for $(f,T)$ to have a periodic maximizing measure is that $T|K(f)$ is injective. 
	
	We introduce the {\em Sturmian-like} property as a generalization of the Sturmian property mentioned before. See the remark after Definition~\ref{defn:sturlike} for a comparison of the Sturmian-like property with similar earlier generalizations of the Sturmian property given in~\cite{Bre06E,Bre06N,HJ07}. We shall show that for ``most" $f\in \mathcal{H}\setminus \mathcal{H}_{\mathrm{per}}$, $T:K(f)\to K(f)$ is Sturmian-like, which means that it has only finitely many branching points and outside these branching points, $T:K(f)\to K(f)$ has a continuous inverse (see Definition~\ref{defn:sturlike}). To this end, we first show that failing to be Sturmian-like implies existence of a {\em supercritical value} (see Definition~\ref{defn:sc}) which satisfies a certain non-transversality condition.  Theorem~\ref{thm:Shy}, together with the result of ~\cite{GS24} mentioned before, then implies that $\mathcal{H}\setminus (\mathcal{H}_{\mathrm{per}}\cup \mathcal{H}_{\mathrm{st}})$ is a countable union of finite-dimensionally shy (see \S~\ref{sse:shy} for its definition) subsets  of $\mathcal{H}$, where 
\begin{equation}\label{eqn:st set}
  \mathcal{H}_{\mathrm{st}}=\{f\in \mathcal{H}: T|K(f) \text{ is Sturmian-like}\}.
\end{equation}
	The proof of Theorem~\ref{thm:Shy} is a modification of an argument of Tsujii~\cite{Tsu01} for a closely related transversality problem in the study of linear skew products over expanding circle maps.
	
	Most of the argument so far works for high-dimensional expanding maps except that the result of ~\cite{GS24} is  for one-dimensional maps. Our last step, to show that $\mathcal{H}_{\mathrm{st}}\setminus \mathcal{H}_{\mathrm{per}}$ is a countable union of finite-dimensionally shy subsets, is more restricted to the one-dimensional case. Our argument is inspired  by the method used to search for the pre-Sturmian condition in~\cite{Bou00}. Exploiting the one-dimensional structure of $\R/\Z$ we show that if $f\in \mathcal{H}_{\mathrm{st}}\setminus \mathcal{H}_{\mathrm{per}}$ has a unique maximizing measure, then $f$ satisfies an identity (see Proposition~\ref{prop:identity}), and the identity can be destroyed by a suitable perturbation (see Theorem~\ref{thm:prevalent1}).

\medskip
	The rest of the paper is organized as follows. In \S~\ref{sec:pre}, we recall some known results which will be used in our argument. We shall use these results to deduce Theorems \ref{thm:A} and \ref{thm:B} from Theorem~\ref{thm:main}. The rest of the paper is devoted to the proof of Theorem~\ref{thm:main}. In \S~\ref{sec:sp}, we introduce the Sturmian-like property and the concept of supercritical value, and discuss their relations. In  \S~\ref{sec:B}, we show that most performance functions do not have a supercritical value. The proof of Theorem~\ref{thm:main} is completed in \S~\ref{sec:final}.

	\medskip
	
	{\bf Acknowledgment.}
	The authors thank Shaobo Gan for valuable discussions, and thank the anonymous referees for valuable suggestions. The work is supported by the National R\&D Program 2021YFA1003200. RG is supported by the Taishan
Scholar Project of Shandong Province (tsqnz20230614) and the National Natural Science Foundation of China (12571201). WS is supported by the New Cornerstone Science Foundation through the New Cornerstone Investigator Program and the XPLORER PRIZE.

	\section{Preliminaries}\label{sec:pre}

	In this section, we review some known facts, including basic facts about sub-actions and Aubry sets. We shall use these facts to deduce Theorems \ref{thm:A} and \ref{thm:B} from Theorem~\ref{thm:main}.

	\subsection{Sub-actions}\label{subsec:subact}
Following~\cite{CLT01}, a continuous function $g:\R/\Z\to \R$ is called a {\em sub-action} of $(f,T)\in \mathcal{C}^0\times \mathcal{E}^1$ if for any $x\in \R/\Z$,
	$$g(x)+f(x)\le g(T(x))+\beta(f,T).$$
	The terminology is borrowed from Lagrangian dynamics.
	A sub-action is called {\em calibrated} if
	$$\sup_{y\in T^{-1}(x)} (g(y)+f(y))= g(x)+\beta(f,T)$$
	for each $x\in \R/\Z$.

	Sub-actions play an important role in the ergodic maximization problem; see for example \cite{Bou00,CLT01,Bou01,Con16,Boc18,Jen19,HLMXZ19}. For $T\in\mathcal{E}^1$, it is well-known that any $f\in\mathcal{C}^1$ has a Lipschitz calibrated sub-action $g$. This is often referred to as {\em Ma\~n\'e's lemma}. In fact, it is well-known that for $(f,T)\in\mathcal{C}^1\times\mathcal{E}^1$, a calibrated sub-action of $(f,T)$ is automatically Lipschitz (see \cite[Lemme~B]{Bou00} or \cite[Proposition~3.3]{Gar17}). Furthermore, there exists a constant $C_0(T)$  such that any (Lipschitz) calibrated sub-action $g$ of $f$ satisfies
	\begin{equation}\label{eqn:Calig}
		\mathrm{Lip}(g)\le C_0(T)\cdot\mathrm{Lip}(f),
	\end{equation}
where $\mathrm{Lip}(\cdot)$ denotes the best Lipschitz constant. Indeed, $C_0(T)$ can be chosen uniformly in a small neighborhood of any $T_0\in \mathcal{E}^1$ (see for example \cite[Lemma~2.1]{Con16}). 	
	
	For a continuous sub-action $g$ of $(f,T)\in \mathcal{C}^1\times \mathcal{E}^1$, let
	\begin{equation}\label{eqn:Efg}
		E_{f,T,g}=\{x\in \mathbb{R}/\mathbb{Z}: g(x)+f(x)=g(T(x))+\beta(f,T)\},
	\end{equation}
and
	\begin{equation}\label{eqn:LambdafTg}
		\Lambda_{f,T, g}=\{x\in E_{f,T,g}: T^n(x)\in E_{f,T,g} \text{ for all } n\ge 0\}.
	\end{equation}
So $T(\Lambda_{f,T,g})\subset \Lambda_{f,T,g}$. Note that $g$ is a calibrated sub-action of $(f,T)$ if and only if $T(E_{f,T,g})=\mathbb{R}/\mathbb{Z}$, and in this situation $T(\Lambda_{f,T,g})=\Lambda_{f,T,g}$.

	Let $G_{f,T}$ denote the collection of all (Lipschitz) calibrated sub-actions $g$ of $(f,T)$. As mentioned above,
	$G_{f,T}\not=\emptyset$ for any $(f,T)\in \mathcal{C}^1\times \mathcal{E}^1$.
	Let
	$$\mathcal{G}=\{(f,T,g): f\in \mathcal{C}^1, T\in \mathcal{E}^1, g\in G_{f,T}\}$$
	be endowed with the topology as a subspace of $\mathcal{C}^1\times \mathcal{E}^1\times \mathcal{C}^0$.
		Let
	\begin{equation}\label{eqn:setK}
		K(f,T)=\bigcup_{g\in G_{f,T}} \Lambda_{f,T,g}. 
	\end{equation}

Let $\mathbf{Cpt}(\R/\Z)$ denote the space of non-empty compact subsets of $\R/\Z$, endowed with the Hausdorff metric. Given a sequence $\{\Lambda_n\}_{n=1}^\infty$ in $\mathbf{Cpt}(\R/\Z)$, we define its {\em Kuratowski limit superior}
	$$\Ls_{n\to\infty} \Lambda_n:=  \{x\in \R/\Z: \exists n_k\to\infty, x_{n_k}\in \Lambda_{n_k}  \text{ such that } x_{n_k}\to x\},$$
	which is an element of $\mathbf{Cpt}(\R/\Z)$.
	A map $F: A\to\mathbf{Cpt}(\R/\Z)$ defined on a metrizable space $A$ is called {\em upper semi-continuous at } $a\in A$ if for $a_n\to a$ in $A$, we have 
	$$\Ls_{n\to\infty} F(a_n)\subset F(a).$$

In the next proposition, we show that the set $K(f,T)$ depends on $(f,T)$ upper semi-continuously. 

	\begin{prop}\label{prop:setK} The following hold.
		\begin{enumerate}
			\item[(i)] The map $(f,T,g)\mapsto \Lambda_{f,T,g}$ is upper semi-continuous from $\mathcal{G}$ to $\mathbf{Cpt}(\R/\Z)$.
			\item[(ii)] $K(f,T)$ is compact and $T(K(f,T))=K(f,T)$ for $(f,T)\in\mathcal{C}^1\times \mathcal{E}^1$.
			\item[(iii)] $(f,T)\mapsto K(f,T)$ is upper semi-continuous from $\mathcal{C}^1\times \mathcal{E}^1$ to $\mathbf{Cpt}(\R/\Z)$.
			\item[(iv)] For each $(f,T)\in \mathcal{C}^1\times \mathcal{E}^1$, the set of invariant Borel probability measures of $T|K(f,T)$ is equal to $\mathcal{M}_{\max}(f,T)$,
			where $$\mathcal{M}_{\max}(f,T)=\left\{\mu\in \mathcal{M}(T): \int f\dif \mu=\beta(f,T)\right\}.$$
		\end{enumerate}
	\end{prop}
	\begin{proof} 
		(i). Suppose $(f_n,T_n,g_n)\to (f,T,g)$ in $\mathcal{G}$ and let $x_n\in \Lambda_{f_n,T_n,g_n}$ be such that $x_n\to x$. Let us show $x\in \Lambda_{f,T,g}$. 
Indeed, for each integer $l\ge 0$ and each $n\ge 1$, $T_n^l(x_{n})\in E_{f_{n}, T_n, g_{n}}$, so
		$$g_{n}(T_n^l(x_{n}))+f_{n}(T_n^l(x_{n}))= g_{n}(T_n^{l+1}(x_{n}))+\beta(f_{n},T_{n}),$$
		and hence
		$$g(T^l(x))+f(T^l(x))=g(T^{l+1}(x))+\beta(f,T).$$
		This implies that $x\in \Lambda_{f,T,g}$.
		
		(ii). Since $T(\Lambda_{f,T,g})=\Lambda_{f,T,g}$ for all $g\in G_{f,T}$, we have $T(K(f,T))= K(f,T)$. To show that $K(f,T)$ is compact, 
let $\{x_n\}\subset K(f,T)$, so for each $n$, $x_n\in \Lambda_{f,T,g_n}$ for some $g_n\in G_{f,T}$. Fixing an arbitrary $y\in \R/\Z$, let $\tilde{g}_n=g_n-g_n(y)$. Then $\tilde{g}_n\in G_{f,T}$ and $\Lambda_{f,T, g_n}=\Lambda_{f,T,\tilde{g}_n}$. Since $\mathrm{Lip}(\tilde{g}_n)$ is bounded in $n$, passing to a subsequence, we may assume that $\tilde{g}_n$ converges to some $g$ in the $C^0$ topology. Then $g\in G_{f,T}$.  By (i), any accumulation point of $\{x_n\}$ belongs to $\Lambda_{f,T,g}\subset K(f,T)$. Thus $K(f,T)$ is compact.
		
		(iii). Suppose $(f_n,T_n)\to (f,T)$ in $\mathcal{C}^1\times\mathcal{E}^1$ and $x_n\in K(f_n,T_n)$, $x_n\to x$. Then for each $n$ there exists $g_n\in G_{f_n,T_n}$ such that $x_n\in \Lambda_{f_n,T_n, g_n}$. 
		As in the proof of (ii),  $x_n\in \Lambda_{f_n,T_n,\tilde{g}_n}$ for $\tilde{g}_n=g_n-g_n(y)$, and passing to a subsequence, we may assume that $\tilde{g}_n$ converges to some $g$ in the $C^0$ topology. Again $g\in G_{f,T}$. By (i), we obtain 
		$x\in \Lambda_{f,T,g} \subset K(f,T)$.
		
		(iv). Any $\mu\in \mathcal{M}_{\max}(f,T)$ is supported on $\Lambda_{f,T,g}$ for any continuous sub-action $g$ of $f$. So $\mu$ is supported on $K(f,T)$, 
		which implies that $\mu$ is an invariant Borel probability measure of $T|K(f,T)$.
		
For the other direction, it suffices to show each ergodic Borel probability measure $\mu$ of $T|K(f,T)$ satisfies $\int f \dif\mu=\beta(f,T)$. 
		By the ergodicity of $\mu$, there exists $x\in K(f,T)$ such that 
		$$\int f\dif\mu=\lim_{n\to\infty} \frac{1}{n}\sum_{i=0}^{n-1} f(T^i(x)).$$
		Since $x\in K(f,T)$, there exists $g\in G_{f,T}$ with $x\in \Lambda_{f,T,g}$. Thus for each $n\ge 1$
		$$\sum_{i=0}^{n-1} f(T^i(x))= n\beta(f,T)+ g(T^n(x))-g(x),$$
		which implies that $\int f\dif\mu=\beta(f,T)$.
	\end{proof}

	\subsection{Shyness and prevalence}\label{sse:shy}
	
	Let $\mathcal{V}$ be a vector space over $\mathbb{R}$ (not necessarily endowed with a topology). For each positive integer $N$, we say that a subset $S$ of $\mathcal{V}$ is {\em $N$-dimensionally shy} 
	if there exist $\varphi_i\in \mathcal{V}$, $i=1,2,\dots, N$, such that 
$$m_N\left(\left\{(t_i)_{i=1}^N\in [0,1]^N: f+ \sum_{i=1}^N t_i\varphi_i\in S\right\}\right)=0$$
	for each $f\in \mathcal{V}$, where $m_N$ denotes the standard Lebesgue (outer) measure on $[0,1]^N$.  
	An $N$-dimensionally shy set in a Banach space is shy (recall its definition given ahead of Theorem~\ref{thm:main}): we simply take $\mu$ to be the pushforward of 
	the standard Lebesgue measure on the cube $[0,1]^N$ to the compact set 
	$$\left\{\sum_{i=1}^N t_i\varphi_i: (t_i)_{i=1}^N\in [0,1]^N\right\}.$$ It is straightforward to check that $\mu$ is transverse to $S$. 
	A subset $S$ of $\mathcal{V}$ is called {\em finite-dimensionally shy} if it is $N$-dimensionally shy for some positive integer $N$.
	
For the proof of Theorem~\ref{thm:main}, we need the following two propositions. 	
	\begin{prop}\label{prop:shycountable} Let $S$ be a countable union of finite-dimensionally shy subsets of a Banach space $\mathcal{V}$. Then there exists $\{\varphi_n\}_{n=1}^\infty\subset \mathcal{V}$ with $\sum_{n=1}^\infty \|\varphi_n\|<\infty$ such that for any $f\in \mathcal{V}$, 
		$$m_\infty\left(\left\{(t_n)_{n=1}^\infty\in [0,1]^{\mathbb{N}_1}: f+\sum_{n=1}^\infty t_n \varphi_n\in S\right\}\right)=0,$$
		where $m_\infty$ denotes the infinite product of the Lebesgue measure on $[0,1]$.
	\end{prop}
	\begin{proof} Let $S_j$, $j=1,2,\dots$ be a sequence of finite-dimensionally shy subsets of $\mathcal{V}$. So for each $j$, there exists 
		$\{\varphi_n^j\}_{n=1}^{N_j}\subset\mathcal{V}\setminus \left\lbrace 0\right\rbrace $, such that for each $f\in \mathcal{V}$,
 $$m_{N_j}\left(\left\{(t_n)_{n=1}^{N_j}\in [0,1]^{N_j}: f+\sum_{n=1}^{N_j} t_n \varphi_n^j\in S_j\right\}\right)=0.$$
		Put $\psi_n^j=\frac{1}{2^jN_j}\varphi_n^j/\|\varphi_n^j\|$. Let $\mathcal{N}=\{(n,j): j\in \mathbb{N}_1, n\in \{1,2,\dots, N_j\}\}$. Then $\sum_{(n,j)\in \mathcal{N}} \|\psi_n^j\|=1$. For each $k=1,2,\dots$, $$m_\infty\left(\left\{(t_n^j)_{n,j}\in [0,1]^\mathcal{N}: f + \sum_{n,j} t_n^j \psi_n^j\in S_k\right\}\right)=0 $$
		by Fubini's theorem.  Thus 
		$$m_\infty\left(\left\{(t_{n}^j)_{n,j}\in [0,1]^\mathcal{N}: f + \sum_{n,j} t_n^j \psi_n^j \in \bigcup_k S_k \right\}\right)=0.$$
	\end{proof}

\begin{prop}\label{prop:shylocal}
Let $\mathcal{V}$ be a topological vector space with a countable topological basis. Let $\mathcal{H}$ be a linear subspace of $\mathcal{V}$ and let $S\subset \mathcal{H}$. Suppose that for each $f\in S$, there is a neighborhood $\mathcal{U}$ of $f$ in $\mathcal{V}$ such that $\mathcal{U}\cap S$ is a countable union of finite-dimensionally shy subsets of $\mathcal{H}$. Then $S$ is a countable union of finite-dimensionally shy subsets of $\mathcal{H}$.  
\end{prop}
	
	\begin{proof} Let $(\mathcal{W}_n)_{n=1}^\infty$ be a countable topological basis for $\mathcal{V}$. Let $E$ be the collection of $n\ge 1$ such that $\mathcal{W}_n\cap S$ is a countable union of finite-dimensionally shy subsets of $\mathcal{H}$. Then 
$$S=\bigcup_{n\in E} (\mathcal{W}_n\cap S),$$ 
and the conclusion follows.
	\end{proof}

	\subsection{The Aubry set}\label{sse:aubry}
	Recall the notation $\Omega(f,T)$ defined ahead of Theorem~\ref{thm:A}, which was introduced in \cite{CLT01} and called the {\em $(f,T)$-non-wandering set} there. We shall call it the {\em Aubry set} of $(f,T)$, following \cite[Definition~4.A]{Gar17}. In~\cite[Proposition~15]{CLT01}, it was proved that for any $(f,T)\in \mathcal{C}^1\times\mathcal{E}^1$, 
	\begin{itemize}
		\item $\Omega(f,T)$ is a non-empty compact set and  $T(\Omega(f,T))=\Omega(f,T)$; 
		\item $\Omega(f,T)$ is contained in $E_{f,T,g}$ for any continuous sub-action $g$ of $(f,T)$ (so $\Omega(f,T)\subset K(f,T)$);
		\item $\Omega(f,T)$ contains the support of each maximizing measure of $(f,T)$.
	\end{itemize} 
	
	For $T\in \mathcal{E}^1$ and a linear subspace $\mathcal{X}$ of $\mathcal{C}^1$, recall $\mathcal{X}_{\#}(T)$ defined in \eqref{eqn:sharp} and let
\begin{equation}\label{eqn:per}
  \mathcal{X}_{\mathrm{per}}(T) =\{f\in \mathcal{X}: (f,T) \text{ has a periodic maximizing measure} \}. 
\end{equation}
By definition, $\mathcal{X}_{\#}(T)\subset \mathcal{X}_{\mathrm{per}}(T)$.

	\begin{thm}\label{thm:permax} Let $T\in \mathcal{E}^1$ and let $\mathcal{X}$ be a linear subspace of $\mathcal{C}^1$ such that all trigonometric polynomials  belong to $\mathcal{X}$. Then we have: 
		\begin{itemize}
			\item [(i)] $\mathcal{X}_{\mathrm{per}}(T)\setminus\mathcal{X}_{\#}(T)$ is a countable union of $1$-dimensionally shy subsets of $\mathcal{X}$;
			\item [(ii)] when $\mathcal{X}$ is a topological vector space, $\mathcal{X}_{\#}(T)$ is dense in $\mathcal{X}_{\mathrm{per}}(T)$.
		\end{itemize}
	\end{thm} 
Note that in assertion (ii), the topology of $\mathcal{X}$ is flexible and it can be irrelevant to the $C^1$ topology.

	\begin{proof} For each $T$-periodic orbit $\mathcal{O}$, let $\mu_{\mathcal{O}}$ denote the averaged Dirac measure  
		on $\mathcal{O}$, and let
		$$\mathcal{X}_{\mathcal{O}}(T)=\{f\in \mathcal{X}: \mu_{\mathcal{O}}\text{ is }(f,T)\text{-maximizing}\}.$$
		To prove (i), since $T$ has only countably many periodic orbits, it suffices to show that for each periodic orbit $\mathcal{O}$ of $T$,
		$\mathcal{X}_{\mathcal{O}}(T)\setminus \mathcal{X}_{\#}(T)$ is a $1$-dimensionally shy subset of $\mathcal{X}$. To prove (ii), it suffices to show that $\mathcal{X}_{\#}(T)\cap \mathcal{X}_{\mathcal{O}}(T)$ is dense in $\mathcal{X}_{\mathcal{O}}(T)$ for each $\mathcal{O}$.
		
To this end, let us follow an argument in \cite{CLT01}. Fix an arbitrary $\mathcal{O}$, and let $\varphi: \R/\Z\to \R$ be a trigonometric polynomial such that $\varphi|\mathcal{O}=0$ and $\varphi(x)<0$ for each $x\in (\R/\Z)\setminus \mathcal{O}$. 
		
		Now let $f_0\in\mathcal{X}$ be arbitrary and define $f_t=f_0+t \varphi$ for $t\in\R$.  
		Then $f_t\in \mathcal{X}$, and let us show that $$\{t\in \R:f_t\in \mathcal{X}_{\mathcal{O}}(T)\setminus \mathcal{X}_{\#}(T)\}$$ contains at most one point.
		Arguing by contradiction, assume that there exist $t_1<t_2$ such that $f_{t_1}, f_{t_2}\in \mathcal{X}_{\mathcal{O}}(T)\setminus \mathcal{X}_{\#}(T).$ Then 
		$$\beta(f_{t_1}, T)=\beta(f_{t_2}, T)=\frac{1}{\#\mathcal{O}}\sum_{p\in \mathcal{O}} f_0(p),$$
where $\#\mathcal{O}$ denotes the cardinality of $\mathcal{O}$. Let $g$ be an arbitrary continuous sub-action of $(f_{t_1},T)$. Then for each $x\in \R/\Z$, 
		$$g(T(x))+\beta(f_{t_1}, T)\ge g(x)+f_{t_1}(x).$$
		Since $\beta(f_{t_2}, T)=\beta(f_{t_1}, T)$ and $f_{t_2}\le f_{t_1}$, 
		$g$ is also a sub-action of $(f_{t_2},T)$. Also note that $f_{t_2}(x)<f_{t_1}(x)$ if $x\notin \mathcal{O}$. It follows that 
		$$\Omega(f_{t_2},T)\subset E(f_{t_2}, T, g)\subset\mathcal{O}.$$ As $\Omega(f_{t_2},T)$ is $T$-invariant and non-empty, this implies that $\Omega(f_{t_2},T)=\mathcal{O}$ is a periodic orbit of $T$, so  $f_{t_2}\in \mathcal{X}_{\#}(T)$, a contradiction. This shows that $\mathcal{X}_{\mathcal{O}}(T)\setminus \mathcal{X}_{\#}(T)$ is $1$-dimensionally shy.
		
		To show that $\mathcal{X}_{\#}(T)\cap \mathcal{X}_{\mathcal{O}}(T)$ is dense in $\mathcal{X}_{\mathcal{O}}(T)$, we observe that if $f_0\in \mathcal{X}_{\mathcal{O}}(T)$, then $f_t\in \mathcal{X}_{\#}(T)\cap \mathcal{X}_{\mathcal{O}}(T)$ for each $t>0$. 
	\end{proof}
	
	\begin{lem}[\cite{CLT01}] \label{lem:omegaK}
		For $(f,T)\in \mathcal{C}^1\times\mathcal{E}^1$, if
\begin{itemize}
  \item [(i)] either $\Omega(f,T)$ is $T$-minimal,
  \item [(ii)] or $(f,T)$ has a unique maximizing measure,
\end{itemize}
then $K(f,T)= \Omega(f,T)$. 
	\end{lem}
	\begin{proof} It suffices to show that for any (Lipschitz) calibrated sub-action $g$ of $(f,T)$, we have $\Omega(f,T)\supset \Lambda_{f,T,g}$. For case (i), since $\Omega(f,T)$ is $T$-minimal, it is $g$-irreducible (see \cite[Definition~13]{CLT01}), so by \cite[Proposition~15~(iii)]{CLT01}, the inclusion follows. For case (ii), it follows from \cite[Proposition 15 (iii)(iv)]{CLT01}.
	\end{proof}

	\begin{lem}\label{lem:CRfper}  
		The set $\{(f,T)\in \mathcal{C}^1\times \mathcal{E}^1: \Omega(f,T) \text{ is a periodic orbit of T}\}$ is an open subset of $\mathcal{C}^1\times \mathcal{E}^1$. 
	\end{lem}

	\begin{proof} Let $(f_0,T_0)\in \mathcal{C}^1\times\mathcal{E}^1$ be such that $K_0=\Omega(f_0,T_0)$ is a $T_0$-periodic orbit. By Lemma~\ref{lem:omegaK}, $K_0=K(f_0, T_0)$. Fix a small neighborhood $V$ of $K_0$ in $\R/\Z$. By the upper semi-continuity of  $(f,T)\mapsto K(f,T)$, there is a neighborhood $\mathcal{W}$ of $(f_0,T_0)$ in $\mathcal{C}^1\times \mathcal{E}^1$ such that $K(f,T)\subset V$ for $(f,T)\in \mathcal{W}$. Shrinking $V$ and $\mathcal{W}$ if necessary, we have that $\bigcap_{n=0}^\infty T^{-n}(V)$ is a periodic orbit of $T$. Since $K(f,T)$ is $T$-invariant, $K(f,T)\subset \bigcap_{n=0}^\infty T^{-n}(V)$. Thus $K(f,T)$ is a $T$-periodic orbit. Since $\Omega(f,T)\subset K(f,T)$,  $\Omega(f,T)$ is a periodic orbit of $T$.
	\end{proof}

From the proof of Lemma~\ref{lem:CRfper} we can immediately obtain the following, which is also a direct consequence of \cite[Proposition~16~(i)]{CLT01}.

\begin{cor}\label{cor:lock}
  For $(f,T)\in \mathcal{C}^1\times\mathcal{E}^1$, if $\Omega(f,T)$ is a $T$-periodic orbit, then there exists a $C^1$ open neighborhood $\mathcal{U}$ of $f$ such that for each $\tilde{f}\in \mathcal{U}$, $\Omega(\tilde{f},T)=\Omega(f,T)$; in particular, $\tilde{f}$ has a unique maximizing measure and it is supported on $\Omega(f,T)$.       
\end{cor}

\subsection{Proof of Theorems \ref{thm:A} and \ref{thm:B}}
	
	We shall now deduce Theorems \ref{thm:A} and \ref{thm:B}, assuming Theorem~\ref{thm:main}.

\begin{proof}[Proof of Theorem~\ref{thm:B} assuming Theorem~\ref{thm:main}]
Let $\mathcal{X}=\mathcal{C}^r$, $r=1,2,\dots,\infty$ and let 
$$\mathcal{E}_*^r=\{T\in\mathcal{E}^r: \mathcal{X}_\#(T) \text{ is open and dense in }\mathcal{X}\}.$$
It suffices to show that $\mathcal{E}_*^r\supset \mathcal{E}^\omega$ and that $\mathcal{E}_*^r$ is a residual subset of $\mathcal{E}^r$. Note that by Lemma~\ref{lem:CRfper}, $\mathcal{X}_\#(T)$ is open in $\mathcal{X}$ for each $T\in \mathcal{X}$. Therefore, 
$$\mathcal{E}_*^r=\{T\in\mathcal{E}^r: \mathcal{X}_\#(T) \text{ is dense in }\mathcal{X}\}.$$
		
Let us first show that $\mathcal{E}^\omega\subset \mathcal{E}_*^r$. Let $T\in \mathcal{E}^\omega$. Let $\mathcal{H}=\mathcal{C}^\omega_{1,1}$ be as defined in \eqref{eqn:analytic space}, which is dense in $\mathcal{X}$. By Theorem~\ref{thm:main}, $\mathcal{H}_{\#}(T)$ is a prevalent subset of $\mathcal{H}$, and hence it is dense in $\mathcal{H}$. Thus $\mathcal{X}_\#(T)\supset \mathcal{H}_\#(T)$ is dense in $\mathcal{X}$.
Therefore $T\in \mathcal{E}_*^r$.

To prove that $\mathcal{E}_*^r$ is a residual subset of $\mathcal{E}^r$, choose $T_n\in \mathcal{E}^\omega$, $n=1,2,\dots$, such that $\{T_n\}_{n=1}^\infty$ is dense in $\mathcal{E}^r$. Then for each $n=1,2,\dots$, $\mathcal{X}_\#(T_n)$ is an open and dense subset of $\mathcal{X}$. Thus $\mathcal{D}_\#=\bigcap_{n=1}^\infty\mathcal{X}_\#(T_n)$ is a residual subset of $\mathcal{X}$. Let $\{f_m\}_{m=1}^\infty$ be a dense subset of $\mathcal{D}_\#$.  By Lemma~\ref{lem:CRfper}, for each $m$, there is an open subset $\mathcal{U}_m$ of $\mathcal{E}^r$ containing $\{T_n\}_{n=1}^\infty$ such that $f_m\in\mathcal{X}_\#(T)$ for each $T\in\mathcal{U}_m$. Let $\mathcal{U}=\bigcap_{m=1}^\infty \mathcal{U}_m$. Then $\mathcal{U}$ is a residual subset of $\mathcal{E}^r$. Now let us show that $\mathcal{U}\subset \mathcal{E}_*^r$. Indeed, for each $T\in \mathcal{U}$, $\{f_m\}_{m=1}^\infty\subset\mathcal{X}_\#(T)$, and hence $\mathcal{X}_\#(T)$ is dense in $\mathcal{X}$. Thus $T\in \mathcal{E}_*^r$, which completes the proof.
\end{proof}

To prove Theorem~\ref{thm:A}, let us fix $r$ and $q$ and denote $\mathcal{X}=\mathcal{C}^r$ below, and let $\mathcal{F}$ be the space of $C^r$ families of functions $f_{\mathbf{t}}:\R/\Z\to \R$ parameterized by $\mathbf{t}\in [0,1]^q$, endowed with the $C^r$ topology. An element $(f_{\mathbf{t}})_{\mathbf{t}\in [0,1]^q}\in \mathcal{F}$ will be denoted by $(f_{\mathbf{t}})$ for short. We need the following lemma to replace the role of Lemma~\ref{lem:CRfper} in the proof of Theorem~\ref{thm:B}.
\begin{lem}\label{lem:leb open}
 For each $\eps\in(0,1)$, the set
\[
\left\{((f_{\mathbf{t}}),T)\in  \mathcal{F}\times \mathcal{E}^r : m_q \big(\{\mathbf{t}\in [0,1]^q: f_{\mathbf{t}} \in \mathcal{X}_{\#}(T)\}\big) >\eps \right\}
\]
is open in  $\mathcal{F}\times \mathcal{E}^r$.
\end{lem}

\begin{proof}
For each $((f_{\mathbf{t}}), T)\in \mathcal{F}\times \mathcal{E}^r$, let 
$$\mathbf{L}((f_{\mathbf{t}}), T)=\{\mathbf{t}\in [0,1]^q: f_{\mathbf{t}} \in \mathcal{X}_{\#}(T)\}.$$
By Lemma~\ref{lem:CRfper}, for any compact subset $K$ of $[0,1]^q$, 
$$\{(\mathbf{f},T)\in \mathcal{F}\times \mathcal{E}^r: \mathbf{L}(\mathbf{f}, T)\supset K\}$$
is an open subset of $\mathcal{F}\times \mathcal{E}^r$. 
 
Now suppose that $m_q\big(\mathbf{L}(\mathbf{f}, T))>\eps$ for some $(\mathbf{f}, T)\in \mathcal{F}\times \mathcal{E}^r$. Then there exists a compact set $K\subset \mathbf{L}(\mathbf{f}, T)$ such that $m_q(K)>\eps$. Thus if $(\tilde{\mathbf{f}}, \tilde{T})$ is sufficiently close to $(\mathbf{f}, T)$ in $\mathcal{F}\times \mathcal{E}^r$, then $\mathbf{L}(\tilde{\mathbf{f}}, \tilde{T}) \supset K$ and hence $m_q(\mathbf{L}(\tilde{\mathbf{f}}, \tilde{T}))>\eps$. The conclusion follows.
\end{proof}

	\begin{proof}[Proof of Theorem~\ref{thm:A} assuming Theorem~\ref{thm:main}] 
For each integer $k\ge 1$ and each $T\in \mathcal{E}^r$, let
		$$\mathcal{F}_k(T)=\left\{(f_{\mathbf{t}})\in \mathcal{F}: m_q\big(\{\mathbf{t}\in [0,1]^q: f_{\mathbf{t}} \in \mathcal{X}_{\#}(T) \}\big)>1-\tfrac{1}{k}\right\},$$
and let 
$$\mathcal{F}_*(T)=\left\{(f_{\mathbf{t}})\in \mathcal{F}: m_q \big(\{\mathbf{t}\in [0,1]^q: f_{\mathbf{t}} \in \mathcal{X}_{\#}(T) \} \big) = 1 \right\}.$$
		Let 
		$$\mathcal{E}_k^r=\{T\in\mathcal{E}^r: \mathcal{F}_k(T) \text{ is a residual subset of } \mathcal{F}\}$$
		for each $k$, and let 
		$$\mathcal{E}_*^r=\{T\in\mathcal{E}^r: \mathcal{F}_*(T) \text{ is a residual subset of } \mathcal{F}\}.$$
		Then for each $T\in\mathcal{E}^r$,
		$$\mathcal{F}_*(T)=\bigcap_{k=1}^\infty \mathcal{F}_k(T).$$
		As a result,
		$$\mathcal{E}_*^r= \bigcap_{k=1}^\infty \mathcal{E}_k^r.$$
		So it suffices to show that for each $k$, $\mathcal{E}_k^r\supset \mathcal{E}^\omega$ and $\mathcal{E}_k^r$ is a residual subset of $\mathcal{E}^r$.

		Let us first show that $\mathcal{E}_k^r \supset \mathcal{E}^\omega$. Given $T\in \mathcal{E}^\omega$, it suffices to show that $\mathcal{F}_k(T)$ is open and  $\mathcal{F}_*(T)$ is dense in $\mathcal{F}$. By Lemma~\ref{lem:leb open}, $\mathcal{F}_k(T)$ is open. 
		It remains to show that $\mathcal{F}_*(T)$ is dense in $\mathcal{F}$. To this end, let $\mathcal{H}=\mathcal{C}^\omega_{1,1}$ be as defined in \eqref{eqn:analytic space}. By Theorem~\ref{thm:main}, there exists $\{\varphi_n\}_{n=1}^\infty\subset \mathcal{H}$, such that $\sum_{n=1}^\infty \|\varphi_n\|_{\mathcal{H}}<\infty$ and such that for each $f\in \mathcal{H}$, 
		$$m_\infty \left(\left\{(t_n)_{n=1}^\infty \in [0,1]^{\mathbb{N}_1}: f+\sum_{n=1}^\infty t_n\varphi_n\not\in \mathcal{H}_\#(T)\right\}\right)=0.$$
		Given any $C^r$ family $(f_{\mathbf{s}})\in \mathcal{F}$, we may approximate it in the $C^r$ topology by a family $(\tilde{f}_{\mathbf{s}})\in \mathcal{F}$ with $\tilde{f}_{\mathbf{s}}\in \mathcal{H}$. Then for each $\mathbf{s}\in [0,1]^q$, 
		$$m_\infty\left(\left\{(t_n)_{n=1}^\infty \in [0,1]^{\mathbb{N}_1}: \tilde{f}_{\mathbf{s}} + \sum_{n=1}^\infty t_n\varphi_n \not\in \mathcal{H}_{\#}(T)\right\}\right)=0.$$
		By Fubini's theorem, for $m_\infty$-a.e. $(t_n)_{n=1}^\infty$, 
		$$m_q\left(\left\{\mathbf{s}\in [0,1]^q: \tilde{f}_{\mathbf{s}} +\sum_{n=1}^\infty t_n\varphi_n\not\in \mathcal{H}_{\#}(T)\right\}\right)=0.$$ 
		For any such $(t_n)_{n=1}^\infty$, $(\tilde{f}_{\mathbf{s}}+ \sum_{n=1}^\infty t_n \varphi_n)$ is a family in $\mathcal{F}_*(T)$. Since the topology on $\mathcal{H}$ is stronger than the $C^r$ topology, this implies that  $\mathcal{F}_*(T)$ is dense in $\mathcal{F}$.

		Now let us turn to the proof that $\mathcal{E}_k^r$ is a residual subset of $\mathcal{E}^r$. The argument is similar to the corresponding part in the proof of Theorem~\ref{thm:B}. To begin with, let $\{T_n\}_{n=1}^\infty\subset \mathcal{E}^\omega$ be dense in $\mathcal{E}^r$. For each $n$, since $T_n\in \mathcal{E}^\omega$, $\mathcal{F}_k(T_n)$ is a residual subset of $\mathcal{F}$. Thus 
		$$\mathcal{F}_k:=\bigcap_{n=1}^\infty \mathcal{F}_k(T_n)$$ 
		is a residual subset of $\mathcal{F}$. Let $\{\mathbf{f}_m\}_{m=1}^\infty$ be a dense subset of $\mathcal{F}_k$. For each $m$, consider
		$$\mathcal{E}^r_{k,m}:=\{T\in \mathcal{E}^r: \mathbf{f}_m \in \mathcal{F}_k(T)\}.$$
		By Lemma~\ref{lem:leb open}, $\mathcal{E}^r_{k,m}$ is open in $\mathcal{E}^r$. On the other hand, since $T_n\in \mathcal{E}^r_{k,m}$ for each $n$, and since $\{T_n\}_{n=1}^\infty$ is dense in $\mathcal{E}^r$, $\mathcal{E}^r_{k,m}$ is dense in $\mathcal{E}^r$. Therefore, $\bigcap_{m=1}^\infty \mathcal{E}^r_{k,m}$ is a residual subset of $\mathcal{E}^r$. For each $T\in \bigcap_{m=1}^\infty \mathcal{E}^r_{k,m}$, $\mathcal{F}_k(T)$ is open in $\mathcal{F}$ and contains $\{\mathbf{f}_m\}_{m=1}^\infty$, so $\mathcal{F}_k(T)$ is a residual subset of $\mathcal{F}$. Thus 
		$$\mathcal{E}^r_k\supset \bigcap_{m=1}^\infty \mathcal{E}^r_{k,m}$$
		is a residual subset of $\mathcal{E}^r$. 
	\end{proof}

\section{Sturmian-like property and Supercritical values}\label{sec:sp}

The rest of the paper is devoted to the proof of Theorem~\ref{thm:main}. In this section, we introduce the notions {\em Sturmian-like} property and {\em supercritical value} and discuss their relations. The main result is Proposition~\ref{prop:scstu}.

In this section, we consider a fixed map $T\in\mathcal{E}^1$. We assume that  
$$T(\pi(0))=\pi(0),$$
where $\pi:\R\to\R/\Z$ is the natural projection. Note that this is not an essential assumption since any expanding circle map can be conjugate to a map fixing $\pi(0)$ via a translation. 

Let $\whT:\mathbb{R}\to\mathbb{R}$ denote the unique lift of $T$ with $\whT(0)=0$ via $\pi$, and let $\tau=\whT^{-1}$. 
Let $d_T\ge 2$ be the absolutely value of the degree of $T$ and let $\A=\{0,1,\dots, d_T-1\}$. For each $i\in \A$ and each $x\in\R$, define
$$
\tau_i(x)=\left\{\begin{array}{ccl}
  \tau(x+i) &,& \text{if $T$ is orientation-preserving} \\
  \tau(x+i) + 1 &,& \text{if $T$ is orientation-reversing}
\end{array}\right..
$$
Note that by definition, 
$$T^{-1}(\pi(x))=\{\pi(\tau_i(x)):i\in\A\}, \quad \forall x\in \R,$$
and
$$\tau_i((0,1))\subset (0,1), \quad \forall i\in\A.$$ 
For each $\mathbf{i}=(i_n)_{n\ge 1}\in \A^{\mathbb{N}_1}$ and each $n\ge 1$, let
\begin{equation}\label{eqn:tauin}
	\tau_{\mathbf{i},n}=\tau_{i_n}\circ \tau_{i_{n-1}}\circ \dots \circ\tau_{i_1}.
\end{equation}
For each H\"older continuous function $f: \R/\Z\to \R$, define 
\begin{equation}\label{eqn:hi}
	h^f_{\mathbf{i}}(x):= \sum_{n=1}^\infty \left(\widehat{f}\circ \tau_{\mathbf{i},n}(x)-\widehat{f}\circ \tau_{\mathbf{i},n}(0)\right), \quad x\in\mathbb{R},
\end{equation}
where $\widehat{f}=f\circ \pi.$ Note that the series on the right hand side converges absolutely. 

\medskip
Let $K$ be a non-empty compact subset of $\R/\Z$ such that $T(K)=K$. Let $\# E$ denote the cardinality of a set $E$ from now on.

\begin{defn}\label{defn:sturlike} 
We say that $x\in K$ is a {\em critical value} of $K$ (or $T|K$), if 
$$\#\left(T^{-1}(x)\cap K\right)>1.$$ 
A critical value $x$ is called {\em regular}, if there exists a connected neighborhood $U$ of $x$ in $\R/\Z$ such that for each connected component $B$ of $U\setminus \{x\}$, the following hold: 
\begin{itemize}
  \item there is a continuous map $\tau_B: B\to\R/\Z$ such that $T\circ \tau_B=\mathrm{id}_B$; 
  \item if $y\in K\cap B$, then $\tau_B(y)$ is the unique point in $T^{-1}(y)\cap K$.
\end{itemize}
We say that $K$ (or $T|K$) is {\em Sturmian-like} if all of its critical values are regular.
\end{defn}

{\em Remark.} Recall that a Sturmian measure $\mu$ is a $T_2$-invariant probability measure supported in a semicircle, where $T_2(x)=2x\mod 1$. These measures appear in Bousch's work \cite{Bou00} on ergodic optimization for $f_t(x)=\cos(2\pi (x-t))$. The support of a Sturmian measure is Sturmian-like for $T_2$ with at most one critical value. As a referee pointed out to us, the Sturmian-like property is  closely related to the notion of ``flower" studied in \cite{Bre06E,Bre06N,HJ07}. A set $F\subset \R/\Z$ is called a {\em flower} of the expanding circle map $T$ if $F$ is a union of finitely many non-degenerate closed arcs, $T(F)=\R/\Z$, and the restriction of $T$ to the interior of $F$ is injective; see \cite[Remark~2.3~(c)]{HJ07}. It can be easily checked that if a non-empty compact set $K$ with $T(K)=K$ is contained in some flower $F$, then $K$ is Sturmian-like, its critical values are contained in $T(\partial F)$ and each critical value has exactly two pre-images in $K$. On the other hand, as shown in ~\cite[Section 2.3]{Bro26}, if $K$ is a Sturmian-like set such that the non-wandering set of $T:K\to K$ is equal to $K$, then $K$ is contained in a flower.  However, when $d_T\ge 3$, in general a Sturmian-like set $K$ may not be contained in any flower, because a critical value of $K$ may have three or more pre-images (see Example~\ref{exa:Sturmian-like}), even though at most two of the pre-images can be limit points of $K$.

\begin{exa}\label{exa:Sturmian-like} 
  Let $T_q(x)=qx \mod 1$ for $q\ge 2$. Given $x_0\in \R/\Z$, let $I_{x_0}$ denote the closed arc $\{x_0+\pi(t):0\le t\le 1/q\}$. Then $S_{x_0}:=\cap_{n=0}^{\infty}T_q^{-n} (I_{x_0})$ is Sturmian-like. Moreover, either $S_{x_0}$ is a periodic orbit, or $S_{x_0}$ has exactly one critical value $T_q(x_0)$. See \cite[Appendix~A]{FSS21} for more details.

Now Suppose $q\ge 3$. Take $x_0$ such that $S_{x_0}$ is not a periodic orbit. Then $c:=T_q(x_0)\in S_{x_0}$.  Choose $c_{-1}\in T^{-1}(c)\setminus S_{x_0}$. Given a fixed point $b$ of $T_q$, we can find a backward orbit $(c_{-n})_{n\ge 1}$ of $T_q$ such that $\lim_{n\to\infty}c_{-n}=b$. Then 
  $$K=S_{x_0}\cup\{c_{-n}: n\ge 1\}\cup\{b\}$$ 
is Sturmian-like and $c$ is a critical value of $K$ having exactly three pre-images in $K$.
\end{exa}

For each $x\in \mathbb{R}$, let $\mathcal{CL}(x, K)$ (``$\mathcal{CL}$" stands for calibrated) denote the collection of $\mathbf{i}\in \A^{\mathbb{N}_1}$ such that
$$\pi(\tau_{\mathbf{i}, n}(x))\in K, \,\, \forall n\ge 1.$$
Note that since $T(K)=K$, $\mathcal{CL}(x,K)\not=\emptyset$ for each $x\in \pi^{-1}(K)$.

\begin{defn}\label{defn:sc}
	For each $f\in \mathcal{C}^1$, we say that $x\in \pi^{-1}(K)$ is an {\em $f$-supercritical value of $K$} if there exist 
	$\mathbf{i}, \mathbf{j}\in \mathcal{CL}(x, K)$ with distinct leading coordinates such that 
	$$ \Dif h^f_{\mathbf{i}}(x)=\Dif h^f_{\mathbf{j}}(x),$$
	where $h^f_{\mathbf{i}}, h^f_{\mathbf{j}}$ are as in \eqref{eqn:hi} and $\Dif$ stands for taking derivative.
\end{defn}

Observe that $x\in\R$ is an $f$-supercritical value of $K$ if and only if so is $x+m$, $m\in \Z$. Therefore it makes sense to say that a point in $K$ is an {\em $f$-supercritical value} of $K$. To verify the observation, it suffices to consider the case $m=1$. First note that there exists a bijection $i\mapsto i'$ from $\A$ to itself such that $\tau_{i'}(x+1)-\tau_i(x)\in \Z$ for any $x\in\R$. 
As a result, given $\mathbf{i}\in\A^{\mathbb{N}_1}$, by induction on $n$, we can find a unique $\mathbf{i}'\in\A^{\mathbb{N}_1}$ such that 
$$\tau_{\mathbf{i}',n}(x+1) - \tau_{\mathbf{i},n}(x) \in \Z, \quad \forall x\in\R, \forall n\ge 1.$$
Then $\Dif h^f_{\mathbf{i}}(x)=\Dif h^f_{\mathbf{i}'}(x+1)$, and $\mathbf{i}\in \mathcal{CL}(x, K)$ if and only if  $\mathbf{i}'\in \mathcal{CL}(x+1, K)$. Moreover, if $\mathbf{i},\mathbf{j}\in\A^{\mathbb{N}_1}$ have distinct leading coordinates, so do the correspondent $\mathbf{i}',\mathbf{j}'$. The conclusion follows.

We shall use repeatedly the following fact (see~\cite[Lemma 3.5]{GS24}). For completeness let us reproduce its proof here. Recall $E_{f,T,g}$ defined by \eqref{eqn:Efg}.
\begin{lem}\label{lem:holonomy}
	Let $g$ be a continuous sub-action of $f\in\mathcal{C}^1$. Let $K$ be a non-empty compact subset of $\R/\Z$ with $T(K)=K\subset E_{f,T,g}$. For any $x\in \pi^{-1}(K)$, any $y\in \R$ and any $\mathbf{i}\in \mathcal{CL}(x,K)$, we have
	\[
	g(\pi(y))-g(\pi(x))\ge h_{\mathbf{i}}^f(y)-h_{\mathbf{i}}^f(x).
	\]
\end{lem}

\begin{proof}
 Since $g$ is a sub-action of $f$, for each $N\ge 1$ we have:
$$g(\pi(y))\ge \sum_{n=1}^N  \widehat{f}(\tau_{\mathbf{i},n}(y)) + g(\pi(\tau_{\mathbf{i},N}(y))).$$
Moreover, since $\mathbf{i}\in \mathcal{CL}(x,K)$, we have:
$$g(\pi(x)) = \sum_{n=1}^N  \widehat{f}(\tau_{\mathbf{i},n}(x)) + g(\pi(\tau_{\mathbf{i},N}(x))).$$
Subtracting the two displayed expressions above first and then letting $N\to \infty$, the conclusion follows.
\end{proof}

In the following $\A$ is endowed with discrete topology and $\A^{\mathbb{N}_1}$ with product topology accordingly; consequently, $\mathcal{CL}(x,K)$ introduced before is always compact. 

\begin{prop}\label{prop:scstu} Let $f,g$ and $K$ be as in Lemma~\ref{lem:holonomy}. If $K$ has no $f$-supercritical value, then $K$ is Sturmian-like.
\end{prop}

\begin{proof} Arguing by contradiction, assume that $K$ is not Sturmian-like. Then there exists $x_0\in \pi^{-1}(K)$ such that it is a critical value of $K$ but not regular. As a result, there exist two sequences $x_n, y_n\in \pi^{-1}(K)$,  both converging to $x_0$ from the same side (say from the right hand side), and distinct $i_1, j_1\in \A$, such that for each $n$,
	$$\tau_{i_1}(x_n), \tau_{j_1}(y_n)\in \pi^{-1}(K).$$
	As $T(K)=K$, there exist 
$$\mathbf{i}^n=(i^n_k)_{k\ge 1}\in \mathcal{CL}(x_n,K) \quad\text{and}\quad
	\mathbf{j}^n=(j^n_k)_{k\ge 1}\in \mathcal{CL}(y_n, K)$$ such that  $i^n_1=i_1\ne j_1=j^n_1$.
	Passing to a subsequence, we may assume $\mathbf{i}^n\to\mathbf{i}$ and $\mathbf{j}^n\to\mathbf{j}$.
	By continuity, $\mathbf{i},\mathbf{j}\in \mathcal{CL}(x_0,K)$.
	By Lemma~\ref{lem:holonomy},
	$$\int_{x_0}^{x_n} \Dif h^f_{\mathbf{i}^n}(t)\dif t\ge g(\pi(x_n))-g(\pi(x_0))\ge \int_{x_0}^{x_n} \Dif h^f_{\mathbf{j}}(t) \dif t,$$
	$$\int_{x_0}^{y_n} \Dif h^f_{\mathbf{j}^n}(t) \dif t\ge g(\pi(y_n))-g(\pi(x_0))\ge \int_{x_0}^{y_n} \Dif h^f_{\mathbf{i}}(t)\dif t.$$
	 Dividing the first inequality by $x_n-x_0$ and then letting $n\to\infty$ gives
	$$\Dif h^f_{\mathbf{i}}(x_0)\ge \Dif h^f_{\mathbf{j}}(x_0).$$
Here we used the fact that as $n\to\infty$, $\Dif h^f_{\mathbf{i}^n}(t)$ converges to $\Dif h^f_{\mathbf{i}}(t)$ locally uniformly in $t$. Similarly, the second gives
	$$\Dif h^f_{\mathbf{j}}(x_0)\ge \Dif h^f_{\mathbf{i}}(x_0).$$
	Thus $$\Dif h^f_{\mathbf{j}}(x_0)= \Dif h^f_{\mathbf{i}}(x_0).$$
Since $i_1\ne j_1$, this implies that $x_0$ is an $f$-supercritical value, a contradiction.
\end{proof}

Recall $K(f,T)$ defined in \eqref{eqn:setK}. We write $K(f)=K(f,T)$ for short hereafter. Also recall that by Proposition~\ref{prop:setK}, for each $f\in\mathcal{C}^1$, $T(K(f))=K(f)$, the maximizing measures of $(f,T)$ are exactly the $T$-invariant Borel probability measures supported on $K(f)$, 
and $K(f)$ depends on $f$ in an upper semi-continuous manner. We say that $f\in\mathcal{C}^1$ is {\em supercritical} if $K(f)$ has an  $f$-supercritical value.

\begin{cor}\label{cor:stscK} Assume that $(f,T)$ has a unique maximizing measure. Then either $K(f)$ is Sturmian-like or $f$ is supercritical. 
\end{cor}
\begin{proof} By Lemma~\ref{lem:omegaK}, $K(f)=\Omega(f,T)$, so for each continuous sub-action $g$ of $(f,T)$, $K(f)\subset E_{f,T,g}$. Alternatively, since $(f,T)$ has a unique maximizing measure, it is well-known that there exists a unique calibrated sub-action $g$ of $(f,T)$ up to an additive constant (see \cite[Lemme~C]{Bou00} or \cite[Proposition~6.7]{Gar17}), so $K(f)=\Lambda_{f,T,g}\subset E_{f,T,g}$ for this $g$. Thus the conclusion follows from Proposition~\ref{prop:scstu} applied to $K=K(f)$.
\end{proof}

Let us end this section with the following remark, and this fact will not be (explicitly) used in this paper.  

\begin{prop} Being not supercritical is an open property in $\mathcal{C}^1$. 
\end{prop}

\newpage
\begin{proof} Suppose that $f_n\to f_0$ in $\mathcal{C}^1$ and for each $n$, $K(f_n)$ has an $f_n$-supercritical value $x_n\in \pi^{-1}(K(f_n))$. We need to show that $f_0$ is  supercritical. Without loss of generality, we may assume $x_n\in [0,1)$. After passing to a subsequence, we may assume $x_n\to x_0$.
	By the upper semi-continuity of $K(\cdot)$, we have $\pi(x_0)\in K(f_0)$. Let $\mathbf{i}^n=(i^n_k)_{k=1}^\infty, \mathbf{j}^n =(j_k^n)_{k=1}^\infty$ be elements in $\mathcal{CL}(x_n, K(f_n))$ such that 
	$i^n_1\not=j^n_1$ and 
\begin{equation}\label{eqn:superfn}
\Dif h^{f_n}_{\mathbf{i}^n}(x_n)=\Dif h^{f_n}_{\mathbf{j}^n}(x_n).
\end{equation}
Passing to a further subsequence, we may assume that $\mathbf{i}^n\to \mathbf{i}=(i_k)_{k=1}^\infty$ and $\mathbf{j}^n\to (j_k)_{k=1}^\infty$ in $\A^{\mathbb{N}_1}$. 
	Then $i_1\not=j_1$. Again by the upper semi-continuity of $K(\cdot)$, we have $\mathbf{i},\mathbf{j}\in \mathcal{CL}(x_0, K(f_0))$. Letting $n\to\infty$ in (\ref{eqn:superfn}), we obtain $$\Dif h^f_\mathbf{i}(x_0)=\Dif h^f_{\mathbf{j}}(x_0),$$
	i.e., $x_0$ is an $f_0$-supercritical value of $K(f_0)$. 
\end{proof}

\section{Shyness of existence of supercritical values}\label{sec:B} 

The goal of this section is to prove the following theorem.

\begin{thm}\label{thm:Shy} 
Let $T\in\mathcal{E}^2$ such that $T(\pi(0))=\pi(0)$.  Fix $C_*>0$ and $\lambda_*>1$ such that \eqref{eqn:C*lambda*} holds. Let $\mathcal{D}$  be a dense subset of $\mathcal{C}^1$ consisting of $C^2$ functions. 
Let $f_0\in\mathcal{C}^2$ be such that 
$$h_{\mathrm{top}} (T|K(f_0))<\tfrac{1}{3}\log \lambda_*$$ 
and such that $K(f_0)$ contains no periodic point. Then there exists an open neighborhood $\mathcal{U}$ of $f_0$ in $\mathcal{C}^2$ and a finite sequence $\{\varphi_j\}_{j=1}^J\subset \mathcal{D}$ such that 
for any $f\in\mathcal{C}^2$,
	$$\left\{(t_1, \dots, t_J)\in [0,1]^J:  f+t_1\varphi_1+\dots+t_J\varphi_J \in \mathcal{U}  \mbox{ is 
			supercritical}\right\}$$
has Lebesgue measure zero. 
\end{thm}

The following is an immediately corollary, where $\mathcal{H}_{\mathrm{per}}=\mathcal{H}_{\mathrm{per}}(T)$ is defined in \eqref{eqn:per}. It is worthy mentioning that the proof of this corollary relies on~\cite{GS24}, where the analyticity of $T$ and $f\in\mathcal{H}$ is essentially used. 
\begin{cor}\label{cor:nosc} 
	Let $T$ and $\mathcal{H}$ be as in Theorem~\ref{thm:main}. Suppose $T(\pi(0))=\pi(0)$ and let
	\begin{equation}\label{eqn:H0}
        \mathcal{H}_{\mathrm{sc}}=\{f\in \mathcal{H}: f\text{ is supercritical}\, \}.
	\end{equation}
	Then $\mathcal{H}_{\mathrm{sc}}\setminus \mathcal{H}_{\mathrm{per}}$ 
	is a countable union of finite-dimensionally shy subsets of $\mathcal{H}$.
\end{cor}
\begin{proof} 
By~\cite{GS24}, for each $f\in \mathcal{H}$, either each $\mu\in \mathcal{M}(T)$ is a maximizing measure of $f$, or each maximizing measure has entropy zero. In the former case, $f\in \mathcal{H}_{\mathrm{per}}$; in the latter case, by  Proposition~\ref{prop:setK} and the variational principle, $h_{\mathrm{top}} (T|K(f))=0$. So for each $f_0\in \mathcal{H}_{\mathrm{sc}}\setminus \mathcal{H}_{\mathrm{per}}$, $h_{\mathrm{top}} (T|K(f_0))=0$. Applying Theorem~\ref{thm:Shy} to $\mathcal{D}=\mathcal{H}$, for $f_0$ as above we can find an open neighborhood $\mathcal{U}$ of $f_0$ in $\mathcal{C}^2$ such that $\mathcal{U}\cap (\mathcal{H}_{\mathrm{sc}}\setminus \mathcal{H}_{\mathrm{per}})$ is finite-dimensionally shy. Then by Proposition~\ref{prop:shylocal} (taking $\mathcal{V}=\mathcal{C}^2$), the corollary follows.
\end{proof}

We shall adopt an argument from Tsujii~\cite{Tsu01} to prove Theorem~\ref{thm:Shy}. For preparation, we need some technical lemmas. To begin with, let us introduce some notations. Recall that $\A=\{0,1,\dots,d_T -1\}$. 

\begin{itemize}
  \item If $\mathbf{i}\in \A^n$ for some $n=1,2,\dots$, then $\mathbf{i}$ is called a {\em finite word} and denote $|\mathbf{i}|=n$. If $\mathbf{i}\in \A^{\mathbb{N}_1}$,  then $\mathbf{i}$ is called an {\em infinite word} and denote $|\mathbf{i}|=\infty$. By saying a {\em word} we mean either a finite or an infinite word. 
  \item For words $\mathbf{i},\mathbf{j},\mathbf{u},\mathbf{v},\dots$, their first letters will be denoted by $i_1,j_1,u_1,v_1,\dots$ respectively.
  \item The notation $\tau_{\mathbf{i},n}$ introduced in \eqref{eqn:tauin} still makes sense for a finite word $\mathbf{i}$ provided that $|\mathbf{i}|\ge n$.
  \item Given $f\in\mathcal{C}^1$, for any word $\mathbf{i}$, as a generalization of \eqref{eqn:hi}, let
  $$	h^f_{\mathbf{i}}:= \sum_{n=1}^{|\mathbf{i}|} \left(\widehat{f}\circ \tau_{\mathbf{i},n}-\widehat{f}\circ \tau_{\mathbf{i},n}(0)\right). $$
  \item For $r=1,2,\dots$ and $f\in \mathcal{C}^r$, $\|f\|_{C^r}:=\max_{0\le k\le r} \|\Dif^k f\|_{C^0}$.
\end{itemize}

\begin{lem}\label{lem:second derivative}
  There exists a constant $C(T)>1$ such that for any $f\in \mathcal{C}^2$ and any word $\mathbf{i}$, the following holds:
$$|\Dif h^f_{\mathbf{i}}(x) - \Dif h^f_{\mathbf{i}}(y)| \le C(T) \cdot \|\Dif f\|_{C^1} \cdot |x-y|, \quad\forall x,y\in\mathbb{R}.$$ 
\end{lem}
\begin{proof}
Direct calculation yields that
 $$ |\Dif h^f_{\mathbf{i}}(x) - \Dif h^f_{\mathbf{i}}(y)| \le \sum_{n=1}^{|\mathbf{i}|} \sup_{z\in \mathbb{R}} |(\widehat{f}\circ\tau_{\mathbf{i},n})''(z)| \cdot |x-y|,$$
and 
$$ |(\widehat{f}\circ\tau_{\mathbf{i},k})''(z)| \le \|\Dif f\|_{C^0}\cdot|\tau_{\mathbf{i},n}''(z)| + \|\Dif^2 f\|_{C^0}\cdot|\tau_{\mathbf{i},n}'(z)|^2.$$
On the other hand, by \eqref{eqn:C*lambda*}, $|\tau_{\mathbf{i},n}'|\le (C_*\lambda_*^n)^{-1}$, and the conclusion follows once we show that $|\tau_{\mathbf{i},n}''|\le C_0 \lambda_*^{-n}$ for some $C_0 >1$ dependent only on $T$. To estimate $|\tau_{\mathbf{i},n}''|$, denote $\mathbf{i}=i_1i_2\dots$ and $\tau_{\mathbf{i},0}=\mathrm{id}$. Then by \eqref{eqn:tauin} and the chain rule, 
$$\tau_{\mathbf{i},n}' = \prod_{m=1}^n \tau_{i_m}'\circ \tau_{\mathbf{i},m-1},$$
and hence
$$\tau_{\mathbf{i},n}'' = \sum_{m=1}^n \frac{\tau_{\mathbf{i},n}'}{\tau_{i_m}'\circ \tau_{\mathbf{i},m-1}}\cdot \left(\tau_{i_m}''\circ \tau_{\mathbf{i},m-1}\right) \cdot \tau_{\mathbf{i},m-1}'.$$
Since $|\tau_{i_m}'|$ is bounded from below, $|\tau_{i_m}''|$ is bounded from above and $|\tau_{\mathbf{i},k}'|\le (C_*\lambda_*^k)^{-1}$, the desired upper bound of $|\tau_{\mathbf{i},n}''|$ follows.
\end{proof}

\begin{lem}\label{lem:complexityapp}
	For any $f_0\in\mathcal{C}^2$ with $\pi(0)\not\in K(f_0)$, there exists $C_0\ge 1$ such that the following holds. Given $\theta>h_{\mathrm{top}} (T|K(f_0))$ and $\delta>0$, there exist $\varepsilon_0\in (0,1)$, a positive integer $N$ and $\mathcal{A}_N\subset \A^N$ with
\begin{equation}\label{eqn:low comp}
  \# \mathcal{A}_N< e^{\theta N},
\end{equation}
which satisfy the properties below. If $f\in\mathcal{C}^2$ is supercritical with 
$$\|f-f_0\|_{C^1}<\varepsilon_0 \quad\text{and}\quad \|f-f_0\|_{C^2}<1,$$ 
then for any positive integer $m$, there exist $\mathbf{u}, \mathbf{v},\mathbf{w}\in \A^{mN}$ with $u_1\not=v_1$ such that:
\begin{equation}\label{eqn:coding diff}
  \left|\Dif h^f_{\mathbf{u}}(\tau_{\mathbf{w},mN}(0))-\Dif h^f_{\mathbf{v}}(\tau_{\mathbf{w},mN}(0))\right|\le C_0\lambda_*^{-mN},
\end{equation}
\begin{equation}\label{eqn:coding base}
 \inf\big\{|\tau_{\mathbf{w},mN}(0)-y|:y\in (0,1)\cap \pi^{-1}(K(f_0))\big\} <\delta,
\end{equation} 
and for $z_1z_2\dots z_{mN} =\mathbf{u},\mathbf{v},\mathbf{w}$,
\begin{equation}\label{eqn:SFT}
  z_{k+1}z_{k+2}\dots z_{k+N}\in\mathcal{A}_N, \quad 0\le k\le (m-1)N.
\end{equation}
\end{lem}

\begin{proof} Fix $\theta>h_{\mathrm{top}} (T|K(f_0))$ and $\delta>0$. For each $n$, let $\mathcal{A}_n$ denote the subset of $\A^n$ consisting of $i_1i_2\dots i_n$ such that
	$$\tau_{i_n}\circ \dots \circ\tau_{i_2}\circ\tau_{i_1} ((0,1))\cap \pi^{-1}(K(f_0))\not=\emptyset.$$ 
Then as $n\to\infty$,
$$\frac{1}{n}\log \#\mathcal{A}_n\to h_{\mathrm{top}} (T|K(f_0))<\theta,$$
	so there exists a positive integer $N$ such that \eqref{eqn:low comp} holds, and such that for
$$V:=\bigcup_{i_1i_2\dots i_N\in \mathcal{A}_N} \tau_{i_N}\circ\dots\circ \tau_{i_2}\circ\tau_{i_1}((0,1)),$$
each connected component of $V$ has length less than $\delta$. Then $\pi(V)$ is an open neighborhood of $K(f_0)$, and by the upper semi-continuity of $f\mapsto K(f)$ (Proposition~\ref{prop:setK}), there exists $\varepsilon_0\in (0,1)$ such that if $\|f-f_0\|_{C^1}<\varepsilon_0$, then
	$K(f)\subset \pi(V)$.

	Now suppose $f$ is as in the lemma and $x_0\in (0,1)\cap \pi^{-1}(K(f))$ is an $f$-supercritical value of $K(f)$. Let $\mathbf{i}=(i_n)_{n\ge 1}, \mathbf{j}=(j_n)_{n\ge 1}\in \A^{\mathbb{N}_1}$ be such that 
	$i_1\not=j_1$, 
	$$\tau_{\mathbf{i},n}(x_0),  \tau_{\mathbf{j},n}(x_0) \in \pi^{-1}(K(f))\mbox{ for all } n\ge 1,$$
	and such that $$\Dif h^f_{\mathbf{i}}(x_0)=\Dif h^f_{\mathbf{j}}(x_0).$$
	It follows that $i_{k+1}i_{k+2}\dots i_{k+N}, j_{k+1}j_{k+2}\dots j_{k+N}\in \mathcal{A}_N$
	for all $k\ge 0$. 

Let $m\ge 1$ be given. Take $\mathbf{u}=i_1i_2\dots i_{mN}$ and $\mathbf{v}=j_1j_2\dots j_{mN}$. Then $u_1\ne v_1$, \eqref{eqn:SFT} holds for $\mathbf{u},\mathbf{v}$ and
	$$\left|\Dif h^f_{\mathbf{u}}(x_0)-\Dif h^f_{\mathbf{v}}(x_0)\right| \le \sum_{k=mN+1}^\infty \left|(\widehat{f}\circ\tau_{\mathbf{i},k})'(x_0) - (\widehat{f}\circ\tau_{\mathbf{j},k})'(x_0)\right|\le C\lambda_*^{-mN}$$
for some constant $C=C(f_0)\ge 1$. On the other hand, since $\pi(x_0)\in K(f)=T(K(f))\subset \pi(V)$, 
there exists $\mathbf{w}\in \A^{mN}$ such that \eqref{eqn:coding base} holds, \eqref{eqn:SFT} holds for $\mathbf{w}$ and 
$$ |\tau_{\mathbf{w},mN}(0)-x_0| \le (C_*\lambda_*^{mN})^{-1}.$$
Combining the last inequality with $\|\Dif f\|_{C^1}<\|\Dif f_0\|_{C^1}+1$ and Lemma~\ref{lem:second derivative} yields that
$$|\Dif h^f_{\mathbf{u}}(\tau_{\mathbf{w},mN}(0))-\Dif h^f_{\mathbf{v}}(\tau_{\mathbf{w},mN}(0))|\le |\Dif h^f_{\mathbf{u}}(x_0)-\Dif h^f_{\mathbf{v}}(x_0)|+\tilde{C}\lambda_*^{-mN}$$
for some $\tilde{C}=\tilde{C}(f_0)$, which proves \eqref{eqn:coding diff} and completes the proof of the lemma.
\end{proof}

\begin{lem}\label{lem:transversal}
Let $f_0\in \mathcal{C}^2$ be such that $K(f_0)$ contains no periodic point. Denote
$$\widehat{K}:=\pi^{-1}(K(f_0))\cap [0,1]\subset (0,1).$$
Let $\mathcal{D}$ be as in Theorem~\ref{thm:Shy}. Then there exist $\delta>0$, $x_1,\dots,x_Q\in \widehat{K}$, and $\varphi_{i,q}\in \mathcal{D}$ for $i\in\A, 1\le q\le Q$, such that the following hold:
\begin{itemize}
  \item [(i)] $\widehat{K}\subset \bigcup_{q=1}^Q(x_q-\delta,x_q+\delta)$;
  \item [(ii)] if $|x-x_q|\le 2\delta$ for some $q$, then for any $i\in \A$ and any word $\mathbf{j}$ we have:
	\begin{equation}\label{eqn:matrix1}
		\left|\Dif h^{\varphi_{i,q}}_{\mathbf{j}}(x)-1\right|\le \frac{1}{3}, \quad\mbox{if}\quad j_1=i; 
	\end{equation}
	and
	\begin{equation}\label{eqn:matrix2}
		\left|\Dif h^{\varphi_{i,q}}_{\mathbf{j}}(x)\right|\le \frac{1}{3}, \quad\mbox{if}\quad j_1\ne i.
	\end{equation}
\end{itemize}
 
\end{lem}

\begin{proof}
Fix an integer $M\ge 2$ with
$$\lambda_*^M \ge \frac{8C_T}{C_*(\lambda_*-1)} \quad\text{for}\quad C_T:=\sup_{x\in \R/\Z}|T'(x)|>1.$$ 
Then there exists $\delta>0$ such that $\bigcup_{z\in \widehat{K}}[z-4\delta,z+4\delta]\subset[0,1]$ and such that the following hold for each $i\in \A$ and each $y\in\widehat{K}$:
\begin{itemize}
  \item $\tau_i([y-4\delta,y+4\delta])\cap \tau_j([y-4\delta,y+4\delta])=\emptyset$ for any $j\in\A\setminus\{i\}$;
  \item  $\tau_i([y-4\delta,y+4\delta])\cap\tau_{\mathbf{j},m}([y-4\delta,y+4\delta])=\emptyset$ for any word $\mathbf{j}$ and any $2\le m\le \min\{M,|\mathbf{j}|\}$. 
\end{itemize}
Note that in the last item we used the assumption that $K(f_0)$ contains no periodic point. By the choice of $\delta$, for each $i\in\A$ and each $y\in \widehat{K}$, we can easily find $\psi_{i,y}\in \mathcal{C}^1$ such that for $\widehat{\psi}_{i,y}=\psi_{i,y}\circ\pi$  we have:
\begin{itemize}
  \item $|(\widehat{\psi}_{i,y})'|\le 2C_T$ on $\R$;
  \item $(\widehat{\psi}_{i,y}\circ\tau_i)'=1$ on $[y-2\delta,y+2\delta]$;
  \item $(\widehat{\psi}_{i,y})'=0$ on $[0,1]\setminus \big(\tau_i([y-4\delta,y+4\delta])\big)$.
\end{itemize}

For each $i\in\A$ and each $y\in \widehat{K}$, by our choices of $\delta$ and $\psi_{i,y}$, when $|x-y|\le 2\delta$, for any word $\mathbf{j}$ we have $$(\widehat{\psi}_{i,y}\circ\tau_{\mathbf{j},m})'(x)=0, \quad 2\le m\le  \min\{M,|\mathbf{j}|\},$$ 
and therefore
$$\big|\Dif h^{\psi_{i,y}}_{\mathbf{j}}(x)-(\widehat{\psi}_{i,y}\circ\tau_{j_1})'(x)\big|
=\Big|\sum_{M<m\le|\mathbf{j}|} (\widehat{\psi}_{i,y}\circ\tau_{\mathbf{j},m})'(x)\Big|\le 2 C_T\sum_{m>M} (C_*\lambda_*^m)^{-1} \le \frac{1}{4}.$$
That is to say, 

$$\big|\Dif h^{\psi_{i,y}}_{\mathbf{j}}(x)-1\big| 
\le \frac{1}{4} \ \ \mbox{if}\ \ j_1=i\,;\quad \big|\Dif h^{\psi_{i,y}}_{\mathbf{j}}(x)\big| 
\le \frac{1}{4} \ \ \mbox{if}\ \ j_1\ne i.$$
By compactness, we can choose $x_q\in \widehat{K},1\le q\le Q$ such that assertion (i) holds. 
Since $\mathcal{D}$ is dense in $\mathcal{C}^1$, for each $i\in \A$ and each $q=1,\dots,Q$, we may choose $\varphi_{i,q}\in \mathcal{D}$ sufficiently close to $\psi_{i,x_q}$ in $C^1$ topology such that both \eqref{eqn:matrix1} and \eqref{eqn:matrix2} hold.
\end{proof}

\begin{proof}[Proof of Theorem~\ref{thm:Shy}] 
Let $f_0$ be as in the statement. Let $\delta$ and $\varphi_{i,q}$ be given by Lemma~\ref{lem:transversal}. Consider the following family of functions: 
	$$\varphi_{\mathbf{t}}=\sum_{i,q} t_{i,q} \varphi_{i,q},\quad \mathbf{t}=(t_{i,q})_{i\in\A, 1\le q\le Q}\in\R^{d_TQ}.$$
Fix  $f_0$ and $\delta$ as above and fix $\theta>0$ such that
$$h_{\mathrm{top}} (T|K(f_0))<\theta<\tfrac{1}{3}\log \lambda_*.$$
Let $C_0,\varepsilon_0,N,\mathcal{A}_N$ be given by Lemma~\ref{lem:complexityapp} accordingly. 
Let $\mathcal{U}$ be an open neighborhood of $f_0$ in $\mathcal{C}^2$ such that for each $f\in \mathcal{U}$,
$$\|f-f_0\|_{C^1}< \eps_0 \quad\text{and}\quad \|f-f_0\|_{C^2}< 1.$$

Our goal is to show that the set
$$\mathbf{S}(F):=\{\mathbf{t}\in [0,1]^{d_TQ}: F+ \varphi_{\mathbf{t}} \in \mathcal{U} \text{ is supercritical}\}$$
has Lebesgue measure zero for any $F\in \mathcal{C}^2$. To this end, fix $F\in \mathcal{C}^2$ and denote $F_{\mathbf{t}}=F+ \varphi_{\mathbf{t}}$ below. We make the following three observations.

Firstly, if $\mathbf{t}\in \mathbf{S}(F)$, then
$$\|F_{\mathbf{t}}-f_0\|_{C^1}<\eps_0 \quad\text{and}\quad \|F_{\mathbf{t}}-f_0\|_{C^2}<1,$$ 
and hence all the statements in Lemma~\ref{lem:complexityapp} hold for $F_{\mathbf{t}}$ instead of $f$. In particular, for any $m\ge 1$, there exist $\mathbf{u}, \mathbf{v},\mathbf{w}$ (depending on $\mathbf{t}$) as in Lemma~\ref{lem:complexityapp} such that 
	\begin{equation}\label{eqn:bc1}
		\left| \Dif h^{F_{\mathbf{t}}}_{\mathbf{u}}(\tau_{\mathbf{w},mN}(0))-\Dif h^{F_{\mathbf{t}}}_{\mathbf{v}}(\tau_{\mathbf{w},mN}(0))\right|\le C_0 \lambda_* ^{-mN}.
	\end{equation}
Moreover, by \eqref{eqn:coding base} and Lemma~\ref{lem:transversal}~(i), 
$|\tau_{\mathbf{w},mN}(0)-x_q|<2\delta$ holds for some $q$. 

Secondly, if $1\le q\le Q$ and $\mathbf{u},\mathbf{v},\mathbf{w}\in \A^{mN}$ satisfy $u_1 \ne v_1$ and  
$|\tau_{\mathbf{w},mN}(0)-x_q|\le 2\delta$, then by Lemma~\ref{lem:transversal}~(ii),  
$$  \left| \Dif h^{\varphi_{u_1,q}}_{\mathbf{u}}(\tau_{\mathbf{w},mN}(0))-\Dif h^{\varphi_{u_1,q}}_{\mathbf{v}}(\tau_{\mathbf{w},mN}(0))\right|\ge \frac{1}{3}.$$
Also note that $\Dif h^{f}_{\bullet}(\cdot)$ is a linear functional of $f$ and $F_{\mathbf{t}}$ is affine in $\mathbf{t}$. Then by Fubini's theorem, it follows that 
$$\{\mathbf{t}\in [0,1]^{d_TQ}: \text{\eqref{eqn:bc1} holds for $\mathbf{t}$}\}$$ 
has Lebesgue measure bounded from above by $6C_0\lambda_* ^{-mN}$. 

Thirdly, by \eqref{eqn:low comp}, the number of possible choices of the triple $(\mathbf{u},\mathbf{v},\mathbf{w})$ satisfying \eqref{eqn:SFT} is bounded from above by 
$$(\#\mathcal{A}_N)^{3m} < e^{3\theta mN}.$$
  
  From the three observations above we conclude that $\mathbf{S}(F)$ has Lebesgue measure bounded by $6C_0(e^{3\theta}\lambda_*^{-1})^{mN}$. Combining this with $\theta<\frac{\log\lambda_*}{3}$ and letting $m\to\infty$, it follows that $\mathbf{S}(F)$ is of Lebesgue measure zero. 
\end{proof}

\section{Shyness of non-periodic Sturmian-like maximization}\label{sec:final}
In this section, we shall complete the proof of Theorem~\ref{thm:main}, and from now on $\mathcal{H}$ denotes a Banach space as in Theorem~\ref{thm:main}. We shall assume without loss of generality that $T\in \mathcal{E}^\omega$ satisfies $T(\pi(0))=\pi(0)$. The main step is to prove the following Theorem~\ref{thm:prevalent}. Recall $K(f)$ defined in \eqref{eqn:setK} and the notations below (here and in the following we drop their dependence on $T$):

\begin{itemize}
  \item $\mathcal{H}_{\#}=\{f\in \mathcal{H}: \text{$\Omega(f,T)$ is a periodic orbit} \}$, see \eqref{eqn:sharp};
  \item $\mathcal{H}_{\mathrm{st}}=\{f\in \mathcal{H}: T|K(f) \text{ is Sturmian-like} \}$, see \eqref{eqn:st set};
  \item $\mathcal{H}_{\mathrm{per}}=\{f\in \mathcal{H}: (f,T) \text{ has a periodic maximizing measure} \}$, see \eqref{eqn:per};
  \item $\mathcal{H}_{\mathrm{sc}}=\{f\in \mathcal{H}: f\text{ is supercritical}\, \}$, see \eqref{eqn:H0}.
\end{itemize}
Moreover,  denote
$$\mathcal{H}_{\mathrm{u}}=\{f\in \mathcal{H}: f \text{ has a unique maximizing measure}\}.$$
The main result in~\cite{Mor21} ensures that $\mathcal{H}\setminus \mathcal{H}_{\mathrm{u}}$ is a shy subset of $\mathcal{H}$.

\begin{thm}\label{thm:prevalent}
$(\mathcal{H}_{\mathrm{u}}\cap \mathcal{H}_{\mathrm{st}})\setminus \mathcal{H}_{\mathrm{per}}$ is a countable union of $1$-dimensionally shy subsets of $\mathcal{H}$.
\end{thm}

\begin{proof}[Proof of Theorem~\ref{thm:main}] By Proposition~\ref{prop:shycountable}, it suffices to show that $\mathcal{H}\setminus \mathcal{H}_{\#}$ is a countable union of finite-dimensionally shy subsets of $\mathcal{H}$. To this end, let us write 
	$$\mathcal{X}_1=\mathcal{H}\setminus \mathcal{H}_{\mathrm{u}}, \quad \mathcal{X}_2=(\mathcal{H}_{\mathrm{u}}\cap \mathcal{H}_{\mathrm{st}})\setminus \mathcal{H}_{\mathrm{per}}, \quad \mathcal{X}_3=\mathcal{H}_{\mathrm{u}}\setminus(\mathcal{H}_{\mathrm{st}}\cup\mathcal{H}_{\mathrm{sc}}),$$  
	$$\mathcal{X}_4=\mathcal{H}_{\mathrm{sc}}\setminus \mathcal{H}_{\mathrm{per}},\quad \mathcal{X}_5=\mathcal{H}_{\mathrm{per}}\setminus \mathcal{H}_\#.$$
Then 
$$\mathcal{H}\setminus \mathcal{H}_{\#} \subset \bigcup_{i=1}^5\mathcal{X}_i,$$
and it remains to prove that for $1\le i\le 5$, $\mathcal{X}_i$ is a countable union of finite-dimensionally shy subsets of $\mathcal{H}$. This is verified as follows. 	
	\begin{itemize}
		\item For $i=1$, in \cite{Mor21} it was implicitly proved that $\mathcal{H}\setminus \mathcal{H}_{\mathrm{u}}$ is a countable union of $1$-dimensionally shy subsets of $\mathcal{H}$. More specifically, given $g\in\mathcal{H}$, let
$$\mathcal{P}_g=\{f\in\mathcal{H}:  s\mapsto \beta(f+sg,T) \text{ is differentiable at $0$} \}.$$
By the convexity of $\beta(\cdot,T)$, the set $\{t\in[0,1]: f+tg \notin \mathcal{P}_g\}$ is countable for any $f\in\mathcal{H}$, so $\mathcal{H}\setminus \mathcal{P}_g$ is a $1$-dimensionally shy subset of $\mathcal{H}$. On the other hand, Morris \cite{Mor21} showed that if $\{g_n\}_{n=1}^\infty\subset \mathcal{H}$ and it is dense in $\mathcal{C}^0$, then $\mathcal{H}_u=\cap_{n=1}^\infty \mathcal{P}_{g_n}$. 
 
		\item For $i=2$, this follows from Theorem~\ref{thm:prevalent}. 
		\item For $i=3$, in fact $\mathcal{X}_3=\emptyset$. This follows from Corollary~\ref{cor:stscK}.
		\item For $i=4$, this follows from Corollary~\ref{cor:nosc}.
		\item For $i=5$, this follows from Theorem~\ref{thm:permax}.
	\end{itemize}
The proof is done.
\end{proof}

In the rest of this section, we shall prove Theorem~\ref{thm:prevalent}. For each $f\in \mathcal{H}_{\mathrm{st}}$, let
$$\Crit(f)=\left\{x\in K(f): \#\big(T^{-1}(x)\cap K(f)\big)>1\right\}$$
be the collection of critical values of $K(f)$, which is a finite set by definition. Also note that if $f\notin \mathcal{H}_{\mathrm{per}}$ additionally, then $\Crit(f)$ is non-empty. This is because, if $\Crit(f)$ is empty, then $T:K(f)\to K(f)$ 
is bijective, and hence it is well-known that $K(f)$ consists of (finitely many) periodic points; see for example \cite{RW04}.  

We shall further decompose $(\mathcal{H}_{\mathrm{u}}\cap \mathcal{H}_{\mathrm{st}})\setminus \mathcal{H}_{\mathrm{per}}$ into a countable union of subsets  
so that in each of the subsets, all performance functions are compatible with a specified ``marked frame", which, roughly speaking, means that 
the dynamics $T: K(f)\to K(f)$ moves ``upper semi-continuously". Motivated by ~\cite{Bou00}, for $f\in (\mathcal{H}_{\mathrm{u}}\cap \mathcal{H}_{\mathrm{st}})\setminus \mathcal{H}_{\mathrm{per}}$, we shall construct a special function $\psi_f$ and show that all such $f$ satisfy an identity of the same type (see Proposition~\ref{prop:identity}). Then we show that functions satisfying such an identity form a $1$-dimensionally shy subset (see Theorem~\ref{thm:prevalent1}).

\subsection{Marked frames}
Let us start with the decomposition that will be specified in \eqref{eqn:Hdecomp}. For any positive integer $n$, let $\mathcal{P}_n$ denote the collection of open arcs of the following form:
$$\pi(\tau_{i_n}\circ\dots\circ \tau_{i_2}\circ\tau_{i_1}(0,1)), \quad i_1i_2\dots i_n\in \A^n,$$
where $\A=\{0,1\dots,d_T-1\}$ as before. For an open subset $U$ of $\R/\Z$, a connected component of $U$ will be called a {\em component} for short.

\begin{defn}\label{defn:frame}
We define a {\em frame} to be a pair $(V, V')$ such that 
\begin{itemize}
	\item there exists a positive integer $N$ such that $V$ is a finite union of elements of $\mathcal{P}_N$;
	\item $V'$ is a subset of $V$ which is a union of components of $T^{-1}(V)$.
\end{itemize}
Given an integer $p\ge 1$, a frame $(V, V')$ is called a {\em marked frame of type $p$} if the following hold. 
\begin{itemize}
	\item There are distinct components $V_1, V_2, \dots, V_p$ of $V$ such that for each $1\le \nu\le p$, $T^{-1}(V_\nu)\cap V'=V'_{\nu,+}\cup V'_{\nu,-}$, where $V'_{\nu,+}$ and $V'_{\nu,-}$ are (not necessarily distinct) elements of $\mathcal{P}_{N+1}$; 
	\item for each component $U$ of $V\setminus \bigcup_{\nu=1}^p V_\nu$, $T^{-1}(U)\cap V'$ is an element of $\mathcal{P}_{N+1}$.
\end{itemize} 
\end{defn}

Given $f\in \mathcal{H}_{\mathrm{st}}\setminus \mathcal{H}_{\mathrm{per}}$, we say that a marked frame $(V, V')$ of type $p$ as above is {\em $f$-compatible} if the following hold.
\begin{itemize}
\item $T(V')\supset K(f)$ and $p=\# \Crit(f)$. 
	\item For each $1\le \nu\le p$, $V_\nu$ contains exactly one point in $\Crit(f)$, and if $x\in K(f)$ 
	is in the right (resp. left) component of $V_\nu\setminus \Crit(f)$, then $(T|V'_{\nu,+})^{-1}(x)$ (resp. $(T|V'_{\nu,-})^{-1}(x)$) 
	is the only point in $T^{-1}(x)\cap K(f)$; moreover, $(T|V'_{\nu,\pm})^{-1}(\Crit(f)\cap V_\nu)\subset K(f)$. 
	\item For each component $U$ of $V\setminus \bigcup_{\nu=1}^p V_\nu$, $T^{-1}(U)\cap K(f)\subset V'$. 
\end{itemize}

\begin{rmk} 
It should be mentioned that in the second item above, $T^{-1}(V_\nu)\cap K(f)$ may not be a subset of $V'$, because the unique point in $V_\nu\cap \Crit(f)$ might have additional pre-image(s) in $K(f)\setminus V'$. Such pre-images are isolated in $K(f)$ and will be ignored in the construction of $\tau^f$ after Lemma~\ref{lem:crit cont}. 

For $T_2(x)=2x\mod 1$, let us give a simple example to explain these definitions. Let $S_{x_0}$ be as given in Example~\ref{exa:Sturmian-like}. Let us choose $x_0=\pi(\widehat{x}_0)$ for some $\widehat{x}_0\in (0,1/4)\setminus\mathbb{Q}$ such that  $T_2(x_0)\in S_{x_0}$. By \cite{Bou00}, $S_{x_0}$ can be realized as $K(f)$ for certain $f\in \mathcal{H}_{\mathrm{st}}\setminus \mathcal{H}_{\mathrm{per}}$, and $T_2(x_0)$ is the unique critical value of $S_{x_0}$.  Then for 
$$V=\pi\left((0, \tfrac{1}{2})\cup (\tfrac{1}{2}, 1)\right), \quad V'=\pi\left((0, \tfrac{1}{4})\cup (\tfrac{1}{4}, \tfrac{1}{2})\cup (\tfrac{1}{2}, \tfrac{3}{4})\right) $$
with 
$$V_1=\pi\left((0, \tfrac{1}{2})\right), \quad V'_{1,+}=\pi\left((0, \tfrac{1}{4})\right), \quad V'_{1,-}=\pi\left((\tfrac{1}{2}, \tfrac{3}{4})\right),$$
$(V,V')$ is a marked frame of type $1$ and it is $f$-compatible.
\end{rmk}

The {\em complexity} $\text{Comp}(V, V')$ of a frame $(V, V')$ is defined as follows. For each positive integer $n$, let 
$$V^n=\{x\in V': T(x), T^2(x), \dots, T^n(x) \in V\}$$ 
and 
$$\text{Comp}(V, V')=\limsup_{n\to\infty} \frac{1}{n} \log \#\{\text{components of } V^n\}.$$

\begin{lem}\label{lem:compatible frame} For each $f\in \mathcal{H}_{\mathrm{st}}\setminus \mathcal{H}_{\mathrm{per}}$ and 
	each $\eps>0$, there is a marked frame that is $f$-compatible with complexity less than $\eps$.
\end{lem}

\begin{proof} For each $x\in \R/\Z$, let $\widehat{x}$ denote the unique point in $\pi^{-1}(x)\cap [0,1)$.
	For each $x\in K(f)\setminus \Crit(f)$, let $i(x)$ denote the unique element in $\A$ such that 
	$$\tau_{i(x)}(\widehat{x})\in \pi^{-1}(K(f)).$$
	Then $x\mapsto i(x)$ is a continuous map from $K(f)\setminus \Crit(f)$ to $\A$, and hence locally constant. Moreover, since $K(f)$ is Sturmian-like, there exists $N$ such that for each $U\in \mathcal{P}_N$, $\#(U\cap \Crit(f))\le 1$, and for each component $B$ of $U\setminus \Crit(f)$, either $B\cap K(f)=\emptyset$ or $x\mapsto i(x)$ is constant on $B\cap K(f)$.
	
	For each $n\ge N$, let $V(n)$ denote the union of elements of $\mathcal{P}_n$ which intersect $K(f)$. For each component $U$ of $V(n)$, define $U'\subset T^{-1}(U)$ as follows.
\begin{itemize}
  \item If $U\cap \Crit(f)=\emptyset$, then take $U'$ as the unique component of $T^{-1}(U)$ that intersects $K(f)$.  
  \item If $U\cap \Crit(f)=\{c\}$ for some $c\in \Crit(f)$, let $U_+$ (resp. $U_-$) be the right (resp. left) component of $U\setminus\{c\}$. For $\alpha\in \{+,-\}$, let $U_\alpha'$ be a component of $T^{-1}(U)$ such that $T^{-1}(U_\alpha\cap K(f))\subset U_\alpha'$ and such that $U_\alpha'\cap K(f)\cap T^{-1}(c)\not=\emptyset$. Note that if $U_\alpha\cap K(f)\not=\emptyset$, then $U_\alpha'$ is uniquely determined, and otherwise there are multiple choices of $U_\alpha'$. Finally, we take $U'=U_+'\cup U_-'$ ($U_+'$ and $U_-'$ may coincide). 
\end{itemize} 
Let $V'(n)$ be the union of all $U'$ defined as above. The frame $(V(n), V'(n))$ is obviously marked so that the resulting marked frame is $f$-compatible. 
	
	Let $s(n)$ denote the complexity of the frame $(V(n), V'(n))$. Fix $\eps>0$. It remains to show that  $s(n)<\eps$ when $n$ is large enough.
	Let $q(n)$ denote the number of components of $V(n)$.  Since $T|K(f)$ has topological entropy zero by~\cite{GS24}, $q(n)< e^{n\eps/2}$ 
	provided that $n$ is large enough. 
	For each $k\ge n$, let
	$$V^k(n)=\{x\in V'(n): T(x), T^2(x), \dots, T^{k}(x)\in V(n)\}.$$
	Write $k=\ell n+r$ with $\ell\ge 1$ and $0\le r<n$. Then 
	 $V^k(n)$ is contained in 
	$$W:=\{x\in V(n): T^{jn}(x)\in V(n) \text{ for } 1\le j\le \ell\},$$
and each component of $W$ contains at most $d_T^r$ components of $V^{k}(n)$. Thus
$$\#\{\text{components of }V^k(n)\} \le d_T^r\cdot \#\{\text{components of } W\}\le d_T^n\cdot q(n)^{\ell+1}.$$
Combining this with $q(n)< e^{n\eps/2}$ and letting $k\to\infty$, it follows that $s(n)<\eps$ provided that $n$ is large enough. 
\end{proof}

Obviously, there are only countably many marked frames. Fix $C_*>0$ and $\lambda_*>1$ satisfying \eqref{eqn:C*lambda*} from now on.  Let $\mathscr{M}$
denote the collection of marked frames $\mathcal{M}$ such that $\mathcal{M}$ is of type $p$ for some $p\ge 1$ and with complexity less than $\tfrac{1}{p}\log\lambda_*$. For each $\mathcal{M}\in \mathscr{M}$, let $\mathcal{H}^{\mathcal{M}}$ denote the collection of $f\in (\mathcal{H}_{\mathrm{u}}\cap \mathcal{H}_{\mathrm{st}})\setminus \mathcal{H}_{\mathrm{per}}$ 
such that $\mathcal{M}$ is $f$-compatible. Then by Lemma~\ref{lem:compatible frame}, 
\begin{equation}\label{eqn:Hdecomp}
	 (\mathcal{H}_{\mathrm{u}}\cap \mathcal{H}_{\mathrm{st}})\setminus \mathcal{H}_{\mathrm{per}} = \bigcup_{\mathcal{M}\in \mathscr{M}}\mathcal{H}^{\mathcal{M}}.
\end{equation}

\subsection{The identity}

Let us fix a marked frame $\mathcal{M}=(V,V')$ of type $p\ge 1$ and with complexity less than $\frac{1}{p}\log\lambda_*$, and follow the notations $V_\nu,V'_{\nu, +}, V'_{\nu, -},1\le \nu\le p$ introduced in Definition~\ref{defn:frame}. Given $f\in \mathcal{H}^{\mathcal{M}}$, for each $1\le \nu\le p$, let $c_\nu(f)$ denote the unique point in $\Crit(f)\cap V_\nu$, and let $V_{\nu, +}$ (resp. $V_{\nu, -}$) 
denote the right (resp. left) component of $V_\nu\setminus \{c_\nu(f)\}$.

\begin{lem}\label{lem:crit cont}
	For each $1\le \nu\le p$, $f\mapsto c_\nu(f)$ is continuous on $\mathcal{H}^{\mathcal{M}}$ with respect to the $C^1$ topology on $\mathcal{H}^{\mathcal{M}}$.
\end{lem}  
\begin{proof} Suppose $f_n,f\in \mathcal{H}^{\mathcal{M}}$ with $f_n\to f$ in $C^1$ topology. For each $1\le \nu\le p$, it suffices to show that $c_\nu(f)$ is the unique limit point of $\{c_\nu(f_n)\}$. For simplicity of notation, we may further assume that $\{c_\nu(f_n)\}$ itself converges. Let
$x_{n,+}$ and $x_{n,-}$ be two distinct points in $T^{-1}(c_\nu(f_n))\cap K(f_n)$ (these points may lie outside $V'$). 
Passing to a subsequence, we may assume that $x_{n,+}\to x_+$ and $x_{n,-}\to x_-$ simultaneously for some $x_+,x_-$. Since $T(x_{n,+})=T(x_{n,-})\in V_\nu$,  $T(x_+)=T(x_-)$ is contained in the closure of $V_\nu$. By the upper semi-continuity of $K(\cdot)$, we have $x_+, x_-\in K(f)$. Since $T$ is locally injective, we have $x_+\ne x_-$.  Note also that $T(K(f))=K(f)$ does not meet the boundary of $V$. It follows that $T(x_+)=T(x_-)\in \Crit(f)\cap V_\nu$ and hence $T(x_+)=T(x_-)=c_\nu(f)$. Consequently, $c_\nu(f_n)\to c_\nu(f)$. 
\end{proof}

We define a continuous map 
$$\tau^f: V\setminus \Crit(f)\to V'$$
such that $T\circ \tau^f= \mathrm{id}$ as follows:
\begin{itemize}
	\item for each $1\le \nu\le p$, $\tau^f(V_{\nu, +})\subset V'_{\nu,+}$ and $\tau^f(V_{\nu,-})\subset V'_{\nu,-}$;
	\item for any $x\in V\setminus (\bigcup_{\nu=1}^p V_\nu)$, $\tau^f(x)$ is the only point in $T^{-1}(x)\cap V'$.
\end{itemize} 
Note that $\tau^f$ is not defined at any $c\in \Crit(f)$; however, the one-sided limits $\lim_{x\to c^+} \tau^f(x)$ and $\lim_{x\to c^-} \tau^f(x)$ exist and lie in $K(f)$. 

From now on for $f\in \mathcal{H}^{\mathcal{M}}\subset \mathcal{H}_{\mathrm{u}}$, let $\mu_f$ denote its unique maximizing measure. Let $\mathrm{supp}(\cdot)$ denote the support of a measure. For distinct $x,y\in \R/\Z$, let $(x,y)$ denote the open arc in $\R/\Z$ defined by $(x,y)=\pi((\widehat{x},\widehat{y}))$ for $\widehat{x}\in \pi^{-1}(x)$ and $\widehat{y}\in \pi^{-1}(y)\cap [\widehat{x},\widehat{x}+1)$.

\begin{lem}\label{lem:crit value cont} Let $f_0\in \mathcal{H}^{\mathcal{M}}$. 
	There exists a neighborhood $\mathcal{U}$ of $f_0$ in $\mathcal{C}^1$, $\nu_0\in \{1,2,\dots, p\}$, positive integers $n_1, n_2$ and 
	continuous maps $a_1, a_2: \mathcal{U}\cap \mathcal{H}^{\mathcal{M}}\to V$ with respect to the $C^1$ topology on $\mathcal{U}\cap \mathcal{H}^{\mathcal{M}}$, such that for any $f\in \mathcal{U}\cap \mathcal{H}^{\mathcal{M}}$ the following hold:
	\begin{itemize}
		\item $(a_1(f),a_2(f))\subset V$ and $(a_1(f),a_2(f))\cap \mathrm{supp}(\mu_{f})\ne\emptyset$;
		\item $T^{n_1}(a_1(f))=T^{n_2}(a_2(f))=c_{\nu_0}(f)$;
		\item For each $0\le k<n_1$ and $0\le l<n_2$, $T^k(a_1(f)), T^l(a_2(f))\in K(f)\setminus \Crit(f)$. 
	\end{itemize}
\end{lem}
\begin{proof}
	Take $\nu_0\in \{1,2,\dots, p\}$ such that $c=c_{\nu_0}(f_0)\in \Crit(f_0)$ is not in the forward orbit of any $c'\in \Crit(f_0)$ under iteration of $T$. 
	Such a $\nu_0$ exists for otherwise, $K(f_0)$ would contain a periodic orbit.  
	Take a backward orbit $(z_{-n})_{n=0}^\infty$ in $K(f_0)$ with $z_0=c$ and $z_{-1}=\lim_{x\to c^+}\tau^{f_0}(x)$. 
	Then $z_{-n}\not\in \Crit(f_0)$ for each $n\ge 1$. We claim that the closure of $\{z_{-n}\}_{n=0}^\infty$ contains $\mathrm{supp}(\mu_{f_0})$. 
	Indeed, any accumulation point of $\frac{1}{n}\sum_{k=0}^{n-1} \delta_{z_{-k}}$ in the weak-* topology is a maximizing measure of $f_0$. 
	Since $\mu_{f_0}$ is the only maximizing measure of $f_0$, the statement follows. We may then choose $n_1, n_2\ge 1$ and $a_i(f_0)=z_{-n_i}$, $i=1,2$, such that the desired properties hold for $f=f_0$ (since $\mu_{f_0}$ has no atom, we can guarantee that $(a_1({f_0}),a_2({f_0}))\cap \mathrm{supp}(\mu_{{f_0}})\ne\emptyset$).
By continuity and Lemma~\ref{lem:crit cont}, the construction can be extended to a $C^1$ neighborhood of $f_0$ in $\mathcal{H}^{\mathcal{M}}$. More precisely, we can define a sequence of $C^1$ neighborhoods $\mathcal{U}_1\supset \mathcal{U}_2\supset\cdots$ of $f_0$ in $\mathcal{H}^{\mathcal{M}}$ and continuous maps $b_n:\mathcal{U}_n \to V$ with $b_n(f_0)=z_{-n}$ for every $n\ge 1$ as follows: first let $b_1(f)=\lim_{x\to c_{\nu_0}(f)^+}\tau^{f}(x)$, which is well-defined and continuous on some $\mathcal{U}_1$ by Lemma~\ref{lem:crit cont} and the construction of $\tau^f$; then define $b_{n+1}(f)=\tau^f(b_n(f))$ inductively on $n$ (shrinking the domain if necessary in this process). As a consequence, $a_i(f):=b_{n_i}(f)$ is continuous on $\mathcal{U}_{n_i}$ for $i=1,2$, and all the desired properties hold by choosing $\mathcal{U}$ small enough. 
\end{proof}

Recall that $d(\cdot,\cdot)$ denotes the standard metric on $\R/\Z$. Shrinking $\mathcal{U}$ in Lemma~\ref{lem:crit value cont} if necessary, there exists $C_0 \ge 1$ such that 
\begin{equation}\label{eqn:af}
\begin{split}
	& \max_{k=0}^{n_1-1} d(T^k(a_1(f_1)), T^k(a_1(f_2)))\le C_0\cdot d(c_{\nu_0}(f_1), c_{\nu_0}(f_2)),\\
	& \max_{l=0}^{n_2-1} d(T^l(a_2(f_1)), T^l(a_2(f_2)))\le C_0 \cdot d(c_{\nu_0}(f_1), c_{\nu_0}(f_2)),  
\end{split}
\end{equation}
for any $f_1, f_2\in \mathcal{U}\cap\mathcal{H}^{\mathcal{M}}$. We shall fix $\mathcal{U}$ such that both Lemma~\ref{lem:crit value cont} and \eqref{eqn:af} hold.

For each $f\in \mathcal{U}\cap \mathcal{H}^{\mathcal{M}}$, let $W_f^0=(a_1(f),a_2(f))$ and define inductively $W_f^k=\tau^f (W_f^{k-1}\setminus \Crit(f))$ for each $k\ge 1$. 
Note that $W_f^k$ is a finite union of open arcs bounded by points in $K(f)$. Let $\ind_E$ denote the indicator function of a set $E$ and define
$$\psi_f=\sum_{n=1}^\infty \ind_{W_f^n}.$$
For $1\le q\le \infty$, let $L^q(\R/\Z)$ denote the space of real-valued $L^q$ (with respect to the Lebesgue measure) functions on $\R/\Z$ and let $\|\cdot\|_q$ denote the associated $L^q$-norm. Let $|\cdot|$ denote the standard Lebesgue measure on $\R/\Z$.

\begin{lem}\label{lem:varphi} For each $f\in \mathcal{U}\cap \mathcal{H}^{\mathcal{M}}$,
	$$\psi_f\in L^q(\R/\Z),\ 1\le q<\infty \quad \text{but} \quad \|\psi_f\|_{\infty}=\infty.$$
\end{lem}
\begin{proof} Since $T$ is expanding and $\tau^f$ is injective, $|W_f^n|$ is exponentially small in $n$, 
	which implies that $\psi_f\in L^q(\R/\Z)$ for each $1\le q<\infty$.
	
Now let us show $\|\psi_f\|_{\infty}=\infty$. Since $\mu_f((a_1(f),a_2(f)))>0$ and $\mu_{f}$ has no atom, by the Poincar\'e recurrence theorem, there exists $x_0\in (a_1(f),a_2(f))\cap K(f)$ such that $x_0$ returns to $(a_1(f),a_2(f))$ infinitely often under iteration of $T$, and $T^n(x_0)\not\in \Crit(f)$ for each $n$. To complete the proof, it suffices to show that for each positive integer $N$, there exists a neighborhood $U_N$ of $x_0$ such that $\psi_f> N$ on $U_N$. To this end, let $0=n_0<n_1<n_2<\cdots$ be nonnegative integers such that $T^{n_j}(x_0)\in (a_1(f), a_2(f))$, $\forall j\ge 0$. By continuity, there exists an open arc $U_N$ containing $x_0$ such that $T^k(U_N)\subset V\setminus \Crit(f)$ for each $0\le k\le n_N$, and $T^{n_i}(U_N)\subset (a_1(f), a_2(f))$ for each $0\le i\le N$. Consequently, $U_N\subset W_f^{n_i}$ for each $0\le i\le N$. Thus $\psi_f>N$ on $U_N$. 
\end{proof}

\begin{prop}\label{prop:identity} For each $f\in \mathcal{U}\cap\mathcal{H}^{\mathcal{M}}$, the following identity holds:
\begin{equation}\label{eqn:identity}
\begin{split}
& \int_{\R/\Z} f'(x)\psi_f(x)  \dif x \\
= & \sum_{k=0}^{n_1-1} f(T^k(a_1(f)))-\sum_{k=0}^{n_2-1} f(T^k(a_2(f)))+(n_2-n_1)\int f \dif\mu_f.
\end{split}
\end{equation}
\end{prop}
\begin{proof} 
	Let $g$ be a Lipschitz sub-action of $f$. Note that by Lemma~\ref{lem:omegaK}~(ii), $K(f)\subset E_{f,T,g}$. We first prove 
	\begin{equation}\label{eqn:gableft}
		g(a_2(f))-g(a_1(f))=\int_{\R/\Z} f'(x)\psi_f(x) \dif x.
	\end{equation}
	To this end, for each open set $U$ of $\R/\Z$ with $U\not=\R/\Z$, let us write it as a disjoint union of open arcs $J_j=(u_j,v_j)$, and define
	$$\mathbf{O}_g(U)=\sum_j (g(v_j)-g(u_j)).$$
	Provided $\partial U\subset K(f)$ ($\partial U$ stands for the boundary of $U$ in $\R/\Z$) and $U\subset V$,
	$$\mathbf{O}_g(U)=\mathbf{O}_g(\tau^f(U\setminus \Crit(f)))+ \int_{\tau^f(U\setminus \Crit(f))} f'(x) \dif x.$$
	To check this, it suffices to consider the case $U$ is connected and it is a routine task.
	
	Therefore, for each $n=1,2,\dots$,
	$$g(a_2(f))-g(a_1(f))=\mathbf{O}_g(W_f^n)+\sum_{i=1}^n \int_{W_f^i} f'(x) \dif x.$$
	As $n\to\infty$, since $g$ is Lipschitz and $T$ is expanding,
$$\mathbf{O}_g(W_f^n)=O(|W_f^n|)\to 0;$$
at the same time, 
$$\sum_{i=1}^n 1_{W_f^i}  \to \psi_f  ~~\text{in}~~ L^1(\R/\Z).$$
Thus \eqref{eqn:gableft} follows.
	
	On the other hand, since the orbits of $a_1(f)$ and $a_2(f)$ are contained in the $g$-action set $E_{f,T,g}$ and since $T^{n_1}(a_1(f))=T^{n_2}(a_2(f))$, we have
	$$g(a_2(f))-g(a_1(f))=\sum_{k=0}^{n_1-1} f(T^k(a_1(f)))-\sum_{k=0}^{n_2-1} f(T^k(a_2(f)))+(n_2-n_1)\int f \dif\mu_f.$$
	Together with \eqref{eqn:gableft}, this implies the identity \eqref{eqn:identity}. 
\end{proof}

\subsection{The perturbation}
\begin{lem}\label{lem:psimove} 
	There exists a constant $C_1 \ge 1$ such that for each $f_1, f_2\in \mathcal{U}\cap \mathcal{H}^{\mathcal{M}}$, 
	$$\|\psi_{f_1}-\psi_{f_2}\|_1\le C_1 \delta  (\log \delta)^2,$$ 
	where $($the right hand side means $0$ when $\delta = 0$$)$

\begin{equation}\label{eqn:delta crit}
  \delta=\delta(f_1, f_2):=\max_{\nu=1}^p d(c_\nu(f_1), c_\nu(f_2))\in [0,\tfrac{1}{2}].
\end{equation}
\end{lem}
\begin{proof}
For $j=1,2$, let $W_j^0=(a_1(f_j), a_2(f_j))$ and for each $k\ge 0$, let
	$$W_{j}^{k+1}=\tau^{f_j}(W_j^k\setminus \Crit(f_j)).$$
	Then  
	\begin{equation}\label{eqn:varphif}
		\|\psi_{f_1}-\psi_{f_2}\|_1\le \sum_{k=0}^\infty |W_1^k\triangle W_2^k|.
	\end{equation}
	Let $A$ denote the union of grand orbits of points in $\Crit(f_1)\cup \Crit(f_2)$ under the iteration of $T$. That is to say, $x\in A$ if and only if there exists $y\in \Crit(f_1)\cup \Crit(f_2)$ and integers $m,n\ge 0$ such that $T^m(x)=T^n(y)$. By definition, $A$ is a countable set. For $j=1,2$ and $k\ge 0$, let $F_j^k=W_j^k\setminus A$, so that  $\tau^{f_1}$ and  $\tau^{f_2}$ are both well-defined on $F_j^k$, and $|F_j^k|=|W_j^k|$. For $1\le \nu\le p$, let $I_\nu\subset V_\nu$ be the (possibly empty) open arc with endpoints $c_\nu(f_1),c_\nu(f_2)$. Since $\tau^{f_1}=\tau^{f_2}$ on $F_2^k\setminus(\cup_{\nu=1}^p I_\nu)$, 
$$|\tau^{f_1}(F_2^k)\triangle \tau^{f_2}(F_2^k)|\le \sum_{\nu=1}^p |\tau^{f_1}(I_\nu)\cup \tau^{f_2}(I_\nu)|\le pC\delta,$$
where $C=2/\min_{x\in \mathbb{R}/\mathbb{Z}}|T'(x)|$. It follows that
	$$|W_1^{k+1}\triangle W_2^{k+1}|\le |\tau^{f_1}(F_1^k\triangle F_2^k)|+|\tau^{f_1}(F_2^k)\triangle \tau^{f_2}(F_2^k)|\le |\tau^{f_1}(F_1^k\triangle F_2^k)|+ pC\delta.$$
Consequently, for each $k\ge 0$ we have:
	$$|W_1^k\triangle W_2^k|\le k p C\delta + |(\tau^{f_1})^k (F_1^0\triangle F_2^0)|.$$
Also note that by \eqref{eqn:af},
$$|F_1^0\triangle F_2^0|=|W_1^0\triangle W_2^0|\le 2C_0\delta.$$
Combining the two displayed lines above with \eqref{eqn:C*lambda*} we have:
	$$|W_1^k\triangle W_2^k|\le (kpC+2C_0(C_*\lambda_*^k)^{-1})\delta,$$
and hence there exists $C'>0$ independent of $\delta$ such that 
$$\sum_{\lambda_*^{-k} \ge \delta} |W_1^k\triangle W_2^k| \le C' \delta  (\log \delta)^2.$$ 
On the other hand,
$$|W_1^k\triangle W_2^k|\le |W_1^k|+|W_2^k|=|(\tau^{f_1})^k(F_1^0)|+|(\tau^{f_2})^k(F_2^0)|\le 2(C_*\lambda_*^k)^{-1},$$
and hence  there exists $C''>0$ independent of $\delta$ such that
$$\sum_{\lambda_*^{-k} < \delta} |W_1^k\triangle W_2^k| \le C'' \delta .$$ 
Combining the two estimates above with \eqref{eqn:varphif}, the lemma follows.
\end{proof}

Due to Proposition~\ref{prop:shylocal} (taking $\mathcal{V}=\mathcal{C}^1$ and $S=\mathcal{H}^{\mathcal{M}}$) and \eqref{eqn:Hdecomp}, Theorem~\ref{thm:prevalent}  follows from the theorem below.

\begin{thm}\label{thm:prevalent1}
	Let $\mathcal{M}$ be a marked frame of type $p\ge 1$ and of complexity less than $\frac{1}{p}\log \lambda_*$. 
	For each $f_0\in\mathcal{H}^{\mathcal{M}}$, there exists $\varphi\in\mathcal{H}$ 
	and an open neighborhood  $\mathcal{U}_{f_0}$ of $f_0$ in $\mathcal{C}^1$ such that for any $F\in\mathcal{H}$, the set
	$\{t\in [0,1]: F+t\varphi\in \mathcal{U}_{f_0}\cap\mathcal{H}^{\mathcal{M}}\}$ has Lebesgue measure zero.
\end{thm}

We need a lemma for the proof of this theorem. 
\begin{lem} Let $f_0\in \mathcal{H}^{\mathcal{M}}$ be as in Theorem~\ref{thm:prevalent1}. Then for each $L\ge 1$, there exists $\varphi \in\mathcal{H}$ such that the following hold: 
\begin{equation}\label{eqn:varphi}
  \|\varphi\|_\infty < 1  \quad\text{and}\quad  \int_{\R/\Z}\varphi'\psi_{f_0} \dif x > L+1.
\end{equation}
\end{lem}

\begin{proof} Since $\mathcal{H}$ is dense in $\mathcal{C}^1$, it suffices to find $\varphi\in \mathcal{C}^1$ satisfying \eqref{eqn:varphi}. To prove this, we first pick a Borel set $J\subset \R/\Z$ with positive Lebesgue measure such that $\psi_{f_0}|J$ is bounded from above by a constant $C>0$. Such a set exists because $\|\psi_{f_0}\|_1<\infty$. Since $\|\psi_{f_0}\|_\infty=\infty$ and $\psi_{f_0}\ge 0$, the set 
$$E=\{x\in \R/\Z: \psi_{f_0}(x) \ge 4L+C+8\}$$ 
has positive Lebesgue measure. Let
$$\rho=\tfrac{1}{4|E|}\cdot 1_E- \tfrac{1}{4|J|} \cdot 1_J \in L^\infty(\R/\Z).$$
Then 
$$\int_{\R/\Z}\rho \dif x =0\,, \quad \|\rho\|_1=\frac{1}{2}\,, \quad \text{and} \quad  \int_{\R/\Z} \rho  \psi_{f_0} \dif x  \ge L+2\,.$$
Now let $\{\rho_n\}_{n=1}^\infty$ to be a sequence in $\mathcal{C}^0$ converging to $\rho$ in $L^2(\R/\Z)$. Let $\varphi_n$ be the function in $\mathcal{C}^1$ such that
$\varphi_n'=\rho_n-\int_{\R/\Z}\rho_n \dif x$ and $\varphi_n(\pi(0))=0$. Then 
$$\|\varphi_n'-\rho\|_1 \le \|\varphi_n'-\rho\|_2 \le  \|\rho_n-\rho\|_2 + \left|\int_{\R/\Z}\rho_n \dif x\right| \to \left|\int_{\R/\Z}\rho \dif x\right|= 0.$$
It follows that
$$\|\varphi_n\|_\infty \le \|\varphi_n'\|_1 \le \|\varphi_n'-\rho\|_1 + \|\rho\|_1 \to  \frac{1}{2},$$
and since $\|\psi_{f_0}\|_2<\infty$,
$$\int_{\R/\Z} \varphi_n' \psi_{f_0} \dif x \ge \int_{\R/\Z} \rho \psi_{f_0} \dif x - \|(\varphi_n'-\rho)\psi_{f_0}\|_1 \ge L+2 - \|\varphi_n'-\rho\|_2\cdot \|\psi_{f_0}\|_2 \to L+2.$$
Therefore,  \eqref{eqn:varphi}  holds by taking $\varphi=\varphi_n$ for a sufficiently large $n$.
\end{proof}

\begin{proof}[Proof of Theorem~\ref{thm:prevalent1}]

Fix $L=2(n_1+n_2+1)$ and let $\varphi\in \mathcal{H}$ be as in \eqref{eqn:varphi}. By Lemma~\ref{lem:crit cont}, Lemma~\ref{lem:psimove} and the upper semi-continuity of $K(\cdot)$, there exists an open neighborhood  $\mathcal{U}_{f_0}\subset \mathcal{U}$ of $f_0$ in $\mathcal{C}^1$ 
and $C_2\ge 1$ such that for any $f\in \mathcal{U}_{f_0}\cap \mathcal{H}^{\mathcal{M}}$,  we have:
	\begin{equation}\label{eqn:integralphipsi}
		\int_{\R/\Z} \varphi'(x)\psi_f(x) \dif x \ge L,
	\end{equation}
and 
	\begin{equation}\label{eqn:f'bounded}
		\|f'\|_\infty\le C_2.
	\end{equation}
	
To proceed, take an arbitrary $F\in \mathcal{H}$, and let $F_t=F+t\varphi\in \mathcal{H},\,t\in\R$.

\begin{clm}
 There exists $C_3\ge 1$ such that if $F_{t_1}, F_{t_2}\in \mathcal{U}_{f_0}\cap\mathcal{H}^{\mathcal{M}}$, then
	\begin{equation}\label{eqn:t1t2}
		|t_1-t_2|\le C_3 \hat{\delta} \quad\text{for}\quad \hat{\delta} :=\delta(F_{t_1},F_{t_2})(\log \delta(F_{t_1},F_{t_2}))^2,
	\end{equation}
	where $\delta(F_{t_1},F_{t_2})$ is defined by \eqref{eqn:delta crit}.
\end{clm}
	
	\begin{proof}[Proof of Claim.]
		By Proposition~\ref{prop:identity}, for $i=1,2$,
	\begin{equation}\label{eqn:identity'}
  \begin{split}
   & \int_{\R/\Z} F_{t_i}'(x)\psi_{F_{t_i}}(x) \dif x \\
  = & \sum_{k=0}^{n_1-1} F_{t_i}(T^k(a_1(F_{t_i})))-\sum_{k=0}^{n_2-1} F_{t_i}(T^k(a_2(F_{t_i})))+(n_2-n_1)\int F_{t_i} \dif\mu_{F_{t_i}}.
  \end{split}
    \end{equation}
Denote
$$\Delta_1= \left|\int F_{t_1}'\psi_{F_{t_1}} \dif x-\int F'_{t_2} \psi_{F_{t_2}} \dif x\right|,$$
$$\Delta_{2,j}=\sum_{k=0}^{n_j-1}|F_{t_1}(T^k(a_j(F_{t_1})))-F_{t_2}(T^k(a_j(F_{t_2})))|, \quad j=1,2,$$
$$\Delta_3 = \left|\int F_{t_1} \dif\mu_{F_{t_1}}-\int F_{t_2} \dif\mu_{F_{t_2}}\right|.$$
Then \eqref{eqn:identity'} implies that
\begin{equation}\label{eqn:identity''}
  \Delta_1 \le \Delta_{2,1} + \Delta_{2,2} + (n_1+n_2)\Delta_3.
\end{equation}
Now let us estimate each term in \eqref{eqn:identity''} separately. 

Firstly, since 
		\begin{align*}
			\Delta_1 \ge {} & \left|\int (F'_{t_1}-F'_{t_2}) \psi_{F_{t_1}} \dif x\right|-\left|\int F'_{t_2}(\psi_{F_{t_1}}-\psi_{F_{t_2}}) \dif x\right| \\
			={} & |t_1-t_2|\left|\int  \varphi'\psi_{F_{t_1}} \dif x\right|-\left|\int F'_{t_2} (\psi_{F_{t_1}}-\psi_{F_{t_2}})\dif x\right|\\
			\ge{} & |t_1-t_2| \left|\int  \varphi'\psi_{F_{t_1}} \dif x\right|-\|F'_{t_2}\|_\infty \|\psi_{F_{t_1}}-\psi_{F_{t_2}}\|_1,
		\end{align*}
		and since $F_{t_1},F_{t_2}\in \mathcal{U}_{f_0}\cap\mathcal{H}^{\mathcal{M}}$, by \eqref{eqn:integralphipsi}, \eqref{eqn:f'bounded} and Lemma~\ref{lem:psimove}, we have
		\begin{equation}\label{eqn:term1}
		\Delta_1 \ge  L|t_1-t_2| -C_1C_2 \hat{\delta}.
		\end{equation}

Secondly, for $j=1,2$ and for each $0\le k< n_j$, 
		\begin{align*}
			& \left|F_{t_1}(T^k(a_j(F_{t_1})))-F_{t_2}(T^k(a_j(F_{t_2})))\right|\\
			\le{} & |F_{t_1} (T^k(a_j(F_{t_1})))-F_{t_2} (T^k(a_j(F_{t_1})))|+ |F_{t_2}(T^k(a_j(F_{t_1})))-F_{t_2}(T^k(a_j(F_{t_2})))|\\
			\le{} & |t_1-t_2|\cdot|\varphi(T^k(a_j(F_{t_1})))|+ \|F'_{t_2}\|_\infty \cdot d\big(T^k(a_j(F_{t_1})),T^k(a_j(F_{t_2}))\big),\\
			\le{} & |t_1-t_2| + C_0 C_2\cdot\delta(F_{t_1},F_{t_2}).
		\end{align*}
Note that in the last step we used 	$F_{t_1},F_{t_2}\in \mathcal{U}_{f_0}\cap\mathcal{H}^{\mathcal{M}}$, $|\varphi|\le 1$, \eqref{eqn:f'bounded} and \eqref{eqn:af}. 
It follows that
		\begin{equation}\label{eqn:term2}
			\Delta_{2,j} \le n_j \left(|t_1-t_2| +  C_0 C_2\cdot\delta(F_{t_1},F_{t_2}) \right), \quad j=1,2.
		\end{equation}

Thirdly,  for $i=1,2$, since $\mu_{F_{t_i}}$ maximizes  $F_{t_i}$ and since $|\varphi|\le 1$, we have: 
		$$\int F_{t_1} \dif\mu_{F_{t_1}}-\int F_{t_2} \dif\mu_{F_{t_2}}
		\le \int (F_{t_1}- F_{t_2}) \dif\mu_{F_{t_1}}= (t_1-t_2) \int \varphi \dif\mu_{F_{t_1}}\le |t_1-t_2|,$$
		and similarly, 
		$$\int F_{t_1} \dif\mu_{F_{t_1}}-\int F_{t_2} \dif\mu_{F_{t_2}}
		\ge \int (F_{t_1}- F_{t_2}) \dif\mu_{F_{t_2}}= (t_1-t_2) \int \varphi \dif\mu_{F_{t_2}}\ge -|t_1-t_2|.$$
Therefore,
		\begin{equation}\label{eqn:term3}
			\Delta_3 \le |t_1-t_2|.
		\end{equation}

Combining the inequalities \eqref{eqn:identity''}, \eqref{eqn:term1}, \eqref{eqn:term2}  and \eqref{eqn:term3} and recalling $L=2(n_1+n_2+1)$, we obtain \eqref{eqn:t1t2}.
	\end{proof}

	Finally, let us show that the set 
	$$S:=\{t\in [0,1] : F_t\in \mathcal{U}_{f_0}\cap \mathcal{H}^{\mathcal{M}}\}$$ has Lebesgue measure zero. 
	
	Let $(V, V')$ be the underlying frame of $\mathcal{M}$. Let $$V^n=\{x\in V': T(x), T^2(x), \dots, T^n(x)\in V\}$$ 
	for each $n$. Let $\mathcal{R}^n$ denote the collection of rectangles $U_1\times U_2\times \dots\times U_p$, 
	where $U_1, U_2, \dots, U_p$ are connected components of $V^n$. As the complexity of $(V,V')$ is less than $\frac{1}{p}\log \lambda_*$, there exist $1<\lambda_2<\lambda_1<\lambda_*$ such that the number of elements in $\mathcal{R}^n$ is at most $\lambda_2^n$, and each of them has diameter at most 
	$\lambda_1^{-n}$, provided that $n$ is large enough. Thus 
	\begin{equation}\label{eqn:Rn}
		\sum_{R\in \mathcal{R}^n}  \diam (R)(\log \diam (R))^2\to 0
	\end{equation}
	as $n\to\infty$. 
	
	Let $\mathbf{c}(F_t)=(c_\nu(F_t))_{\nu=1}^p$ for $t\in S$. Note that
	$\mathbf{c}(F_t)$ is contained in a rectangle in $\mathcal{R}^n$. On the other hand, by \eqref{eqn:t1t2},
	if $\mathbf{c}(F_{t_1}), \mathbf{c}(F_{t_2})$ are contained in the same  element $R$ of 
	$\mathcal{R}^n$, then 
	$$|t_1-t_2|\le C_3 \cdot \diam (R)(\log \diam (R))^2.$$
	Combining this with \eqref{eqn:Rn}, the statement follows.
\end{proof}

\bibliographystyle{plain}

\bibliography{referTPO}

\end{document}